\documentclass[11pt,fleqn]{article}

\makeatletter

\def\blfootnote{\gdef\@thefnmark{}\@footnotetext}

\makeatother

\usepackage{amsmath,amssymb,amsthm}
\usepackage{geometry} 
\usepackage[pdfpagemode=UseNone,pdfstartview=FitH,hypertex]{hyperref}
\usepackage{enumitem}

\newcommand{\abs}[1]{\left|#1\right|}
\newcommand{\bdry}[1]{\partial #1}

\newcommand{\A}{{\cal A}}
\newcommand{\D}{{\cal D}}
\newcommand{\F}{{\cal F}}

\newcommand{\dist}[2]{\text{dist}\, (#1,#2)}

\newcommand{\incl}{\hookrightarrow}
\renewcommand{\L}{{\cal L}}
\newcommand{\M}{{\cal M}}
\newcommand{\N}{\mathbb N}
\newcommand{\norm}[2][]{\left\|#2\right\|_{#1}}

\newcommand{\pnorm}[2][]{\if #1'' \left|#2\right|_p \else \left|#2\right|_{#1} \fi}
\newcommand{\R}{\mathbb R}
\newcommand{\RP}{\R \text{P}}
\newcommand{\seq}[1]{\left(#1\right)}
\newcommand{\set}[1]{\left\{#1\right\}}

\newcommand{\w}[1]{\widetilde{#1}}
\newcommand{\weak}{\rightharpoonup}
\newcommand{\Z}{\mathbb Z}

\newcommand{\om}{\int_{\mathbb{R}^N}}

\newcommand{\dev}[1]{\int_{\mathbb{R}^N}|\nabla#1|^2dx}
\newcommand{\sca}[1]{\int_{\mathbb{R}^N}|#1|^qdx}
\newcommand{\cri}[1]{\int_{\mathbb{R}^N}|#1|^{2^*}dx}
\newcommand{\p}{q_{\textrm{rad}}}

\newcommand{\e}{E^{\alpha,\,p}_{\text{rad}}(\R^N)}
\newcommand{\col}[1]{\int_{\mathbb{R}^N}|I_{\alpha/2}\star|#1|^p|^2dx}

\newenvironment{enumroman}{\begin{enumerate}

}{\end{enumerate}}

\newtheorem{corollary}{Corollary}[section]
\newtheorem{lemma}[corollary]{Lemma}
\newtheorem{proposition}[corollary]{Proposition}
\newtheorem{theorem}[corollary]{Theorem}

\theoremstyle{definition}
\newtheorem{definition}[corollary]{Definition}

\theoremstyle{remark}
\newtheorem{example}[corollary]{Example}
\newtheorem{remark}[corollary]{Remark}

\numberwithin{equation}{section}

\title{\bf New solutions to Schr\"{o}dinger--Poisson--Slater equations in Coulomb-Sobolev spaces}

\author{\bf Artur Jorge Marinho\\
Department of Mathematics\\
Florida Institute of Technology\\
150 W University Blvd, Melbourne, FL 32901-6975, USA\\
\em amarinho2024@my.fit.edu\\
[\medskipamount]
\bf Carlo Mercuri\\
Dipartimento di Scienze Fisiche, Informatiche e Matematiche\\
Universit\`{a} di Modena e Reggio Emilia\\
Via Campi 213/b, 41125 Modena, Italy\\
\em carlo.mercuri@unimore.it\\
[\medskipamount]
\bf Kanishka Perera\\
Department of Mathematics\\
Florida Institute of Technology\\
150 W University Blvd, Melbourne, FL 32901-6975, USA\\
\em kperera@fit.edu}
\date{}

\begin{document}

\maketitle

\begin{abstract}
\noindent We prove existence and multiplicity results for the nonlinear and nonlocal PDE
\[
- \Delta u + (I_\alpha \star |u|^p)\, |u|^{p-2}\, u = f(|x|,u) \quad \text{in } \R^N,
\]
where $N \ge 2$, $I_\alpha : \R^N \setminus \set{0} \to \R$ is the Riesz potential of order $\alpha \in (1,N), p>1,$ and the local nonlinearity $f: [0,\infty) \times \R \rightarrow \R$ is subject to a new class of assumptions. We find solutions to this zero-mass problem in a Coulomb–Sobolev space using a new scaling based approach in critical point theory, by which we classify the possibly different behaviour of the nonlinearity $f$ at zero and at infinity in terms of the scaling properties of the left hand side of the equation. This is accomplished identifying a scaling invariant PDE which can be interpreted as a nonlinear eigenvalue problem, for which a sequence of eigenvalues $\{\lambda_k\}$ is conveniently defined via the $\mathbb{Z}_2$–cohomological index of Fadell and Rabinowitz. This index allows us to use new critical group estimates (and scaling-based linking sets) which might not be possible via the classical genus. 
Within a fairly broad set of parameters $N,\alpha, p$ and class of assumptions on the local nonlinearity $f,$ we establish compactness results for an associated action functional and find multiple solutions as critical points, whose existence and number is sensitive to the ``resonance'' of $f$ with the sequence of eigenvalues for the scaling invariant problem, a construction which is at places reminiscent, in the present nonlinear setting, of the classical Fredholm alternative. As a byproduct of our analysis, letting $p\neq 2$ allows us to capture general nonlinearities $f$ of Sobolev-subcritical, critical, or supercritical growth.
\end{abstract}

\newpage

{\small \tableofcontents}

\newpage

\section{Introduction}

In this paper we prove existence and multiplicity results for the nonlinear and nonlocal PDE
\begin{equation} \label{101}
- \Delta u + (I_\alpha \star |u|^p)\, |u|^{p-2}\, u = f(|x|,u) \quad \text{in } \R^N,
\end{equation}

where a new and broad class of assumptions on the Carath\'{e}odory function $f: [0,\infty) \times \R \rightarrow \R$ is considered. 

Here $N \ge 2, p>1$, $I_\alpha : \R^N \setminus \set{0} \to \R$ is the Riesz potential of order $\alpha \in (1,N)$ defined by
\[
I_\alpha(x) = \frac{A_\alpha}{|x|^{N - \alpha}}, \qquad A_\alpha = \frac{\Gamma\! \left(\frac{N - \alpha}{2}\right)}{\Gamma\! \left(\frac{\alpha}{2}\right) \pi^{N/2}\, 2^\alpha},
\]
and the choice of the constant $A_\alpha$ ensures that the kernel $I_\alpha$ enjoys the semigroup property
\[
I_{\alpha+\beta}=I_\alpha\star I_\beta
\]
for each $\alpha,\beta\in (0,N)$ with $\alpha+\beta<N$, see for example Du Plessis \cite[pg. 73-74]{du1970introduction}.
Hence the corresponding (Hartree) energy term in the equation can be conveniently written, and in certain instances treated, as an $L^2$-like term
\begin{equation}\label{Hartreeterm}
A_\alpha \int_{\R^N} \int_{\R^N} \frac{|u(x)|^p\, |u(y)|^p}{|x - y|^{N - \alpha}}\, dx\, dy= \int_{\R^N} \abs{I_{\alpha/2} \star |u|^p}^2\, dx.
\end{equation}

Using the language of the celebrated work of Berestycki and Lions \cite{MR695535,MR695536}, equation \eqref{101} may be regarded as a zero-mass problem, which is naturally set, unlike its positive-mass variants, in a new {\it Coulomb-Sobolev} space which in general does not coincide with the classical Sobolev space $H^1(\R^N)$; see Mercuri et al. \cite{MR3568051}. \\

When $N=3, \alpha=p=2,$ \eqref{101} appears in the literature, with a number of interesting variants, as an approximation of the Hartree--Fock model for a quantum many-body system of electrons, see e.g. Le Bris and Lions \cite{MR2149087} and references therein, and arise, as highlighted by Sanchez and Soler \cite{MR2032129} in models for semiconductors. We shall refer to equation \eqref{101}, with our ``extended'' range of parameters, as {\em Schr\"odinger--Poisson--Slater's} (Bokanowski et al.\! \cite{MR2013491}), although variants of it may be known under different names, such as {\em Schr\"odinger--Poisson's} (Ambrosetti \cite{MR2465993}), {\em Schr\"odin\-ger--Poisson--$X_\alpha$'s} (Mauser \cite{MR1836081}), or {\em Maxwell--Schr\"o\-din\-ger--Poisson's} (Benci and Fortunato \cite{MR1659454} and Catto et al.\! \cite{MR3078678}). The function $|u|^2 : \R^3 \to \R$ typically represents the density of electrons in the original many-body system, whereas the Coulomb-Riesz term represents the Coulombic repulsion between the electrons. The local term $f(|x|,u)=|u|^{q-2}\, u$ was introduced by Slater \cite{Slater} with $q = 8/3$ as a local approximation of the exchange potential in the Hartree--Fock model (see, e.g., Bokanowski et al.\! \cite{MR2013491}, Bokanowski and Mauser \cite{MR1702877}, and Mauser \cite{MR1836081}). \\

When $N=3, \alpha=p=2,$ \eqref{101} is also related to the Thomas--Fermi--Dirac--von\;Weizs\"acker (TFDW) model of Density Functional Theory (DFT); see, e.g., Lieb \cite{ MR629207,MR641371}, Le Bris and Lions \cite{MR2149087}, Lu and Otto \cite{MR3251907}, Frank et al.\! \cite{MR3762278}, and references therein. In the context of the TFDW model, the local nonlinearity takes the form $f(|x|,u) = u^{5/3} - u^{7/3}$, for which recently Mercuri and Perera \cite{MePe2} have proved the existence of infinitely many solutions at negative energy levels, a phenomenon which could be regarded as complementary to that highlighted in \cite{MR3251907}, which instead provides nonexistence of least energy solutions with a prescribed large mass. \newline 

The mathematical choice of considering in equation \eqref {101} either a more general local nonlinearity $f(|x|,u)$ and set of parameters $N,\alpha, p,$ is motivated by physical considerations again arising from DFT and quantum chemistry. In fact, typical models in this area, such as Kohn-Sham's, involve local nonlinearities $f$ which are not explicitly known; see, e.g., Anantharaman and Canc\`es \cite{MR2569902} and references therein. Moreover, as recently highlighted in Lu et al. \cite{MR3415051}, some models may require in the energy term \eqref{Hartreeterm} a choice of the parameters different from $N=3, \alpha=p=2.$

Allowing $N,\alpha, p$ to vary in a certain admissible range, produces various critical numbers which are responsible for the loss of local and global compactness properties of embeddings into Lebesgue spaces. As shown in \cite{MR3568051}, this is the case of
\begin{itemize}
\item  a {\it Coulomb-Sobolev} exponent $q_{CS}=2(2p+2)/(2+\alpha)$ which together with the classical Sobolev exponent is a threshold for optimal embeddings into $L^q(\R^N);$
\item $q_{\textrm{rad}}=2\, [2p\, (N - 1) + N - \alpha])/(3N + \alpha - 4),$ which characterises for $\alpha>1$ the embeddings of the subspace of radial functions into $L^q(\R^N);$ see Theorem \ref{Theorem 1} below;
\item the exponent $p_{CS}=(N+\alpha)/(N-2),$ which (provided $p<p_{CS}$ or $p>p_{CS}$) determines, when $N\geq 3,$ whether the above thresholds lie below or above the classical Sobolev exponent $2^*=2N/(N-2);$
\item the exponent $q^{\textrm{loc}}_{CS}=2\alpha/(\alpha-2),$ which governs, when $\alpha>2,$ the compactness of the relevant embeddings into $L^p_{\textrm{loc}}(\R^N)$ and the validity of a nonlocal Brezis-Lieb property of the term \eqref{Hartreeterm}.

\end{itemize}

This rich functional analytical and variational scenario not yet explored in full, even when $f(|x|,u)=|u|^{q-2}u$ allows dramatic changes in the geometry of the energy functional associated with \eqref{101}, which may exhibit, according to the different values of $q,$ either a coercive, or a mountain-pass like behavior (see e.g. \cite[Remark 7.1]{MR3568051}); or, the nonstandard structure (in terms of the Euler-Lagrange equation) of a {\it nonlinear eigenvalue problem}. The latter has been highlighted first by Ianni and Ruiz \cite{MR2902293} in the case $N=3, \alpha=p=2, q=3,$ as a consequence of the invariance of \eqref{101} by the scaling $u_t=t^2 u(t \,\cdot\,).$  \\
Our main contribution is to extend some recent existence and multiplicity results of Mercuri and Perera \cite{MePe2}, which deal with the case $N=3, \alpha=p=2,$ adapting the scaling based variational approach developed in \cite{MePe2} to situations in which the mountain-pass geometry, linking geometry based on linear subspaces, or $\Z_2$ symmetry may not be present, making most of the available variational techniques inspired by the classical work of Ambrosetti and Rabinowitz \cite{MR370183}, not directly applicable. Roughly speaking, following \cite{MePe2}, we are able to classify the asymptotic/variational features of the nonlinearity $f$ in terms of the aforementioned eigenvalue problem, a reasoning which is reminiscent, in our nonlinear context, of the classical Fredholm alternative.
We construct minimax eigenvalues of this eigenvalue problem using the $\Z_2$-cohomological index of Fadell and Rabinowitz \cite{MR0478189}, a choice which allows us to use new critical group estimates (and scaling-based linking sets) which might not be possible using the classical genus. Our analysis is complementary to that recently provided in a Sobolev-subcritical regime by Gloss et al. \cite{Gloss}, where a fractional variant of \eqref{101} has been studied in the case $p=2.$ In some of our results we use the assumption $2p<2^*$ to guarantee the boundedness of the Palais-Smale sequences, a condition we hope to relax in future. As a byproduct of our approach (and of the functional setting with $p\neq 2$), we are then able to capture nonlinearities $f$ of Sobolev-subcritical, critical, or supercritical growth.  However, as in  \cite{MePe2}, in our construction we will replace this well-established terminology with terms like ``subscaled'', ``scaled'', and ``superscaled'' to reflect more effectively and classify the various variational regimes which may occur, depending on $f$ and on the values of our parameters.

\section{Existence and multiplicity results in Coulomb-Sobolev spaces}

We seek solutions to our PDE \eqref{101} in the so-called Coulomb-Sobolev space $E^{\alpha,\,p}(\R^N)$ consisting of weakly differentiable functions $u : \R^N \to \R$ for which the norm
\[
\norm{u} := \left[\int_{\R^N} |\nabla u|^2\, dx + \left(\int_{\R^N} \abs{I_{\alpha/2} \star |u|^p}^2\, dx\right)^{1/p}\right]^{1/2}
\]
is finite. We work in the subspace $E^{\alpha,\,p}_{\text{rad}}(\R^N)$ of radial functions in $E^{\alpha,\,p}(\R^N)$. The special case $N = 3$, $\alpha = 2$, and $p = 2$ of this space was studied in Lions \cite{MR636734,MR879032} and Ruiz \cite{MR2679375}, where it was shown that $E^{2,\,2}_{\text{rad}}(\R^3)$ is a separable and uniformly convex Banach space that is embedded in $L^q(\R^3)$ continuously for $q \in (18/7,6]$ and compactly for $q \in (18/7,6)$. The general case $N \ge 2$, $\alpha \in (1,N)$, and $p > 1$ was studied in Mercuri et al.\! \cite{MR3568051}, from which we recall and use the following embedding result.

Set $2^*=2N/(N-2)$ whenever $N\geq3$, and $2^*=\infty$ for $N=2$.

\begin{theorem}[{\cite[Theorems 4 \& 5]{MR3568051}}] \label{Theorem 1}
Let $N \ge 2$, $\alpha \in (1,N)$, $p > 1$, and
\[
q_{\textrm{rad}} = \frac{2\, [2p\, (N - 1) + N - \alpha]}{3N + \alpha - 4}.
\]
\begin{enumroman}
\item\label{i1} If $1/p > (N -2)/(N + \alpha)$, then $E^{\alpha,\,p}_{\text{rad}}(\R^N)$ is embedded in $L^q(\R^N)$ continuously for 
\[
\dfrac{1}{2}-\dfrac{1}{N}\leq\dfrac{1}{q}<\dfrac{1}{\p}
\]

 and compactly if the inequalities are strict.
\item\label{i2} If $1/p < (N-2)/(N + \alpha)$, then $E^{\alpha,\,p}_{\text{rad}}(\R^N)$ is embedded in $L^q(\R^N)$ continuously for 
\[
\dfrac{1}{2}-\dfrac{1}{N}\geq\dfrac{1}{q}>\dfrac{1}{\p}
\]
and compactly if the inequalities are strict.
\end{enumroman}
\end{theorem}
\begin{remark}
An interesting feature of our functional setting is that $E^{\alpha,\,p}_{\text{rad}}(\R^N)\not\subset L^{q_{\textrm{rad}}}(\R^N);$ see \cite{MR2679375, MR3568051} .
\end{remark}

In view of Theorem \ref{Theorem 1}, a natural growth condition for the local nonlinearity $f$ is
\begin{equation} \label{102}
|f(|x|,t)| \le a_1\, |t|^{q_1 - 1} + a_2\, |t|^{q_2 - 1} + a(x) \quad \text{for a.a.\! } x \in \R^N \text{ and all } t \in \R,
\end{equation}
where $a_1, a_2 > 0$, and $a\in L^r(\R^N)$ with either
\[
\dfrac{1}{2}-\dfrac{1}{N}\leq\dfrac{1}{q_2}<\dfrac{1}{q_1}<\dfrac{1}{\p}
\]
if $1/p>(N-2)/(N+\alpha)$, and some $r>1$ where
\begin{enumerate}[label=(\alph*)]
    \item $r\in(1,\p/(\p-1))$ if $N=2$,
    \item $r\in [2^\ast/(2^\ast - 1),q_{\textrm{rad}}/(q_{\textrm{rad}} - 1))$ if $N\geq3$
\end{enumerate}
or 
\[
\dfrac{1}{2}-\dfrac{1}{N}\geq\dfrac{1}{q_2}>\dfrac{1}{q_1}>\dfrac{1}{\p}
\]
if $1/p<(N-2)/(N+\alpha)$ and some $r>1$ where $(q_{\textrm{rad}}/(q_{\textrm{rad}} - 1),2^\ast/(2^\ast - 1)]$.

 Under this growth assumption, the associated energy functional
\[
\Phi(u) = \frac{1}{2} \int_{\R^N} |\nabla u|^2\, dx + \frac{1}{2p} \int_{\R^N} \abs{I_{\alpha/2} \star |u|^p}^2\, dx - \int_{\R^N} F(|x|,u)\, dx, \quad u \in E = E^{\alpha,\,p}_{\text{rad}}(\R^N),
\]
where $F(|x|,t) = \int_0^t f(|x|,\tau)\, d\tau$ is the primitive of $f$, is well-defined, of class $C^1$, and its critical points are the weak solutions to equation \eqref{101}.

The nonlinear eigenvalue problem
\begin{equation} \label{103}
- \Delta u + (I_\alpha \star |u|^p)\, |u|^{p-2}\, u = \lambda\, |u|^{q_{CS}-2}\, u \quad \text{in } \R^N,
\end{equation}
will play a key role in our approach to existence and multiplicity for equation \eqref{101}. In what follows we simply set $$q = q_{CS}=2\, (2p + \alpha)/(2 + \alpha).$$ This eigenvalue problem is not only nonlinear, but also nonhomogeneous in the sense that the three terms that appear in it contain different powers of $u$. However, it has a certain scaling invariance property. In fact, let $\sigma = (2 + \alpha)/2(p - 1)$ and consider the scaling
\[
E^{\alpha,\,p}_{\text{rad}}(\R^N) \times [0,\infty) \to E^{\alpha,\,p}_{\text{rad}}(\R^N), \quad (u,t) \mapsto u_t
\]
defined as
\[
u_t(x) = t^\sigma u(tx)
\]
if $1/p > (N-2)/(N+\alpha)$ and
\[
u_t(x) = \begin{cases}
t^{- \sigma} u(x/t) & \text{if } t > 0\\[5pt]
0 & \text{if } t = 0
\end{cases}
\]
if $1/p < (N-2)/(N+\alpha)$. Setting
\[
\L(u) = - \Delta u + (I_\alpha \star |u|^p)\, |u|^{p-2}\, u, \qquad \L_\lambda(u) = \L(u) - \lambda\, |u|^{q-2}\, u,
\]
we have
\[
\L_\lambda(u_t)(x) = \begin{cases}
t^{(4p + \alpha - 2)/2(p - 1)}\, \L_\lambda(u)(tx) & \text{if } 1/p > (N-2)/(N+\alpha)\\[10pt]
t^{- (4p + \alpha - 2)/2(p - 1)}\, \L_\lambda(u)(x/t) & \text{if } 1/p < (N-2)/(N+\alpha)
\end{cases}
\]
for $t > 0$, so if $u$ is a solution of the equation \eqref{103}, then the entire $1$-parameter family of functions $\set{u_t : t \ge 0}$ are, too. We shall refer to it as a scaled eigenvalue problem. We note that $q \in (q_{\textrm{rad}},2^\ast)$ if $1/p > (N-2)/(N+\alpha)$ and $q \in (2^\ast,q_{\textrm{rad}})$ if $1/p < (N-2)/(N+\alpha)$. We will see that the first eigenvalue $\lambda_1$ is positive and there exists an unbounded sequence of eigenvalues
\[
\lambda_1 \le \lambda_2 \le \cdots
\]
(see Theorem \ref{Theorem 301}).

The scaled eigenvalue problem \eqref{103} leads to a natural classification of the local nonlinearity $f$ in equation \eqref{101} in terms of its scaling properties. Assume that the limit
\[
l_\infty = \lim_{|t| \to \infty}\, \frac{f(|x|,t)}{|t|^{q-2}\, t}
\]
exists uniformly a.e. If $1/p > (N-2)/(N+\alpha)$ and
\[
\dfrac{1}{2}-\dfrac{1}{N}\leq\dfrac{1}{r}<\dfrac{1}{\p}
\]
 then
\[
\big(\L(u_t) - |u_t|^{r-2}\, u_t\big)(x) = t^{(4p + \alpha - 2)/2(p - 1)}\, \big(\L(u) - t^{(r-q) \sigma}\, |u|^{r-2}\, u\big)(tx),
\]
so we will say that
\begin{enumroman}
\item $f$ is subscaled if $l_\infty = 0$,
\item $f$ is asymptotically scaled if $l_\infty \in \R \setminus \set{0}$,
\item $f$ is superscaled if $l_\infty = \pm \infty$.
\end{enumroman}
On the other hand, if $1/p < (N-2)/(N+\alpha)$ and $r \in [2^\ast,q_{\textrm{rad}})$, then
\[
\big(\L(u_t) - |u_t|^{r-2}\, u_t\big)(x) = t^{- (4p + \alpha - 2)/2(p - 1)}\, \big(\L(u) - t^{(q-r) \sigma}\, |u|^{r-2}\, u\big)(x/t)
\]
for $t > 0$, so we will say that
\begin{enumroman}
\item $f$ is subscaled if $l_\infty = \pm \infty$,
\item $f$ is asymptotically scaled if $l_\infty \in \R \setminus \set{0}$,
\item $f$ is superscaled if $l_\infty = 0$.
\end{enumroman}
We observe and elaborate on it in the sequel that, as the variational features of the energy functional $\Phi$ associated with the equation \eqref{101} are sensitive to the various asymptotic regimes of $f$ classified above, we will apply suitable (and new) minimax principles, accordingly. Our results include sufficient conditions for the existence of nontrivial solutions in the case $f(|x|,0) \equiv 0$.
Let us now state our main existence and multiplicity results, whose proof is presented in Section \ref{Section 4} after a section on the relevant preliminaries in critical point theory (Section 3).

\subsection{The case \texorpdfstring{$1/p > (N-2)/(N+\alpha)$}{1/p > (N-2)/(N+alpha)}}

First we consider the case $1/p > (N-2)/(N+\alpha)$ and assume that $f$ satisfies the growth condition \eqref{102} with 
\[
\dfrac{1}{2}-\dfrac{1}{N}\leq\dfrac{1}{q_2}<\dfrac{1}{q_1}<\dfrac{1}{\p}.
\]
 We will consider the subscaled, asymptotically scaled, superscaled, and Sobolev critical cases. 

\subsubsection{Subscaled case}

In the subscaled case we have the following existence result.

\begin{theorem} \label{Theorem 2}
Let $1/p > (N-2)/(N+\alpha)$. Assume that $f$ satisfies \eqref{102} with 
$q_{\textrm{rad}} < q_1 < q_2 < q$. Then equation \eqref{101} has a solution.
\end{theorem}

In the next theorem we establish some sufficient conditions for the existence of nontrivial solutions when
\begin{equation} \label{104}
f(|x|,t) = \lambda\, |t|^{q-2}\, t + g(|x|,t),
\end{equation}
where $g$ is a Carath\'{e}odory function on $[0,\infty) \times \R$ satisfying
\begin{equation} \label{105}
|g(|x|,t)| \le a_3\, |t|^{q_3 - 1} + a_4\, |t|^{q_4 - 1} \quad \text{for a.a.\! } x \in \R^N \text{ and all } t \in \R
\end{equation}
for some $a_3, a_4 > 0$ and $q < q_3 < q_4 < 2^\ast$.

\begin{theorem} \label{Theorem 3}
Let $1/p > (N-2)/(N+\alpha)$. Assume that $f$ satisfies \eqref{102} with $q_{\textrm{rad}} < q_1 < q_2 < q$, \eqref{104} and \eqref{105} hold, and that $\lambda$ is not an eigenvalue of problem \eqref{103}.
\begin{enumroman}
\item If $\lambda > \lambda_1$, then equation \eqref{101} has a nontrivial solution.
\item If $\lambda > \lambda_2$, then equation \eqref{101} has two nontrivial solutions.
\end{enumroman}
\end{theorem}

The assumption that $\lambda$ is not an eigenvalue can be removed when the primitive $G(|x|,t) = \int_0^t g(|x|,\tau)\, d\tau$ of $g$ is negative for $t \in \R \setminus \set{0}$.

\begin{theorem} \label{Theorem 4}
Let $1/p > (N-2)/(N+\alpha)$. Assume that $f$ satisfies \eqref{102} with $q_{\textrm{rad}} < q_1 < q_2 < q$, \eqref{104} and \eqref{105} hold, and that
\begin{equation} \label{106}
G(|x|,t) < 0 \quad \text{for a.a.\! } x \in \R^N \text{ and all } t \in \R \setminus \set{0}.
\end{equation}
\begin{enumroman}
\item\label{Theorem 4-1} If $\lambda > \lambda_1$, then equation \eqref{101} has a nontrivial solution.
\item\label{Theorem 4-2} If $\lambda > \lambda_2$, then equation \eqref{101} has two nontrivial solutions.
\end{enumroman}
\end{theorem}

For example, consider the equation
\begin{equation} \label{107}
- \Delta u + (I_\alpha \star |u|^p)\, |u|^{p-2}\, u = \frac{\lambda\, |u|^{q-2}\, u}{1 + |u|^\beta} \quad \text{in } \R^N,
\end{equation}
where $0 < \beta < \min \set{q - q_{\textrm{rad}},2^\ast - q}$. The nonlinearity $f(t) = \lambda\, |t|^{q-2}\, t/(1 + |t|^\beta)$ satisfies
\[
|f(t)| \le |\lambda|\, |t|^{q - \beta - 1} \quad \forall t \in \R
\]
and $g(t) = f(t) - \lambda\, |t|^{q-2}\, t = - \lambda\, |t|^{q + \beta - 2}\, t/(1 + |t|^\beta)$ satisfies
\[
|g(t)| \le |\lambda|\, |t|^{q + \beta - 1} \quad \forall t \in \R.
\]
So the following corollary is immediate from Theorem \ref{Theorem 4}.

\begin{corollary}
Let $1/p > (N-2)/(N+\alpha)$.
\begin{enumroman}
\item If $\lambda > \lambda_1$, then equation \eqref{107} has a nontrivial solution.
\item If $\lambda > \lambda_2$, then equation \eqref{107} has two nontrivial solutions.
\end{enumroman}
\end{corollary}

\subsubsection{Asymptotically scaled case}

In the asymptotically scaled case we consider the equation
\begin{equation} \label{108}
- \Delta u + (I_\alpha \star |u|^p)\, |u|^{p-2}\, u = \lambda\, |u|^{q-2}\, u + g(|x|,u) \quad \text{in } \R^N,
\end{equation}
where $g$ is a Carath\'{e}odory function on $[0,\infty) \times \R$ satisfying
\begin{equation} \label{109}
|g(|x|,t)| \le a_3\, |t|^{q_3 - 1} + h(x) \quad \text{for a.a.\! } x \in \R^N \text{ and all } t \in \R
\end{equation}
for some $a_3 > 0$, $q_{\textrm{rad}} < q_3 < q$, and $h \in L^\infty(\R^N) \cap L^r(\R^N)$ with $r\in(1,\p/(\p-1))$ if $N=2$ and $r \in [2^\ast/(2^\ast - 1),q_{\textrm{rad}}/(q_{\textrm{rad}} - 1))$ if $N\geq3$. We have the following existence result.

\begin{theorem} \label{Theorem 5}
Let $1/p > (N-2)/(N+\alpha)$. Assume that \eqref{109} holds and that $\lambda$ is not an eigenvalue of problem \eqref{103}. Then equation \eqref{108} has a solution.
\end{theorem}

In particular, we have the following

\begin{corollary}
Let $1/p > (N-2)/(N+\alpha)$. Then either the equation
\[
- \Delta u + (I_\alpha \star |u|^p)\, |u|^{p-2}\, u = \lambda\, |u|^{q-2}\, u \quad \text{in } \R^N
\]
has a nontrivial solution, or the equation
\[
- \Delta u + (I_\alpha \star |u|^p)\, |u|^{p-2}\, u = \lambda\, |u|^{q-2}\, u + h(x) \quad \text{in } \R^N
\]
has a solution for all $h \in L^\infty(\R^N) \cap L^r(\R^N)$ with $r\in(1,\p/(\p-1))$ if $N=2$ and $r \in [2^\ast/(2^\ast - 1),q_{\textrm{rad}}/(q_{\textrm{rad}} - 1))$ if $N\geq3$.
\end{corollary}

\subsubsection{Superscaled case}

In the superscaled case we also assume that $1/p > (N-2)/N$, so that $2p < 2^\ast$. Noting that $q < 2p$, we consider the equation
\begin{equation} \label{110}
- \Delta u + (I_\alpha \star |u|^p)\, |u|^{p-2}\, u = \lambda\, |u|^{q-2}\, u + g(|x|,u) + |u|^{r-2}\, u \quad \text{in } \R^N,
\end{equation}
where $2p < r < 2^\ast$ and $g$ is a Carath\'{e}odory function on $[0,\infty) \times \R$ satisfying
\begin{equation} \label{111}
|g(|x|,t)| \le a_3\, |t|^{q_3 - 1} + a_4\, |t|^{q_4 - 1} \quad \text{for a.a.\! } x \in \R^N \text{ and all } t \in \R
\end{equation}
for some $a_3, a_4 > 0$ and $q < q_3 < q_4 < r$. Condition \eqref{111} implies that $g(|x|,0) \equiv 0$ and hence equation \eqref{110} has the trivial solution $u = 0$. We will show that this equation has a nontrivial solution when the primitive $G(|x|,t) = \int_0^t g(|x|,\tau)\, d\tau$ of $g$ is nonnegative.

\begin{theorem} \label{Theorem 6}
Let $1/p > (N-2)/N$. Assume that \eqref{111} holds, $\lambda$ is not an eigenvalue of problem \eqref{103}, and that
\begin{equation} \label{112}
G(|x|,t) \ge 0 \quad \text{for a.a.\! } x \in \R^N \text{ and all } t \in \R.
\end{equation}
Then equation \eqref{110} has a nontrivial solution.
\end{theorem}

The sign condition \eqref{112} can be removed by introducing a small parameter in front of the term $g(|x|,u)$.

\begin{theorem} \label{Theorem 8}
Let $1/p > (N-2)/N$. Assume that \eqref{111} holds and that $\lambda$ is not an eigenvalue of problem \eqref{103}. Then there exists a $\mu_0 > 0$ such that the equation
\[
- \Delta u + (I_\alpha \star |u|^p)\, |u|^{p-2}\, u = \lambda\, |u|^{q-2}\, u + \mu\, g(|x|,u) + |u|^{r-2}\, u \quad \text{in } \R^N
\]
has a nontrivial solution for all $|\mu| < \mu_0$.
\end{theorem}

For example, consider the equation
\begin{equation} \label{113}
- \Delta u + (I_\alpha \star |u|^p)\, |u|^{p-2}\, u = \lambda\, |u|^{q-2}\, u + \mu\, |u|^{\beta - 1} + |u|^{r-2}\, u \quad \text{in } \R^N,
\end{equation}
where $2p < r < 2^\ast$ and $q < \beta < r$. Note that it does not follow from standard arguments involving index theories that this equation has infinitely many solutions since the right-hand side is not an odd function of $u$. The following corollary is immediate from Theorem \ref{Theorem 8}.

\begin{corollary}
Let $1/p > (N-2)/N$. Then there exists a $\mu_0 > 0$ such that the equation \eqref{113} has a nontrivial solution for all $\lambda \in \R$ and $|\mu| < \mu_0$.
\end{corollary}

\subsubsection{Critical case}

Now we assume $N\geq3$. In this case we again assume that $p < N/(N - 2)$, so that $2p < 2^\ast$, and consider the equation
\begin{equation} \label{114}
- \Delta u + (I_\alpha \star |u|^p)\, |u|^{p-2}\, u = \lambda\, |u|^{q-2}\, u + \mu\, |u|^{\beta - 2}\, u + |u|^{2^\ast - 2}\, u \quad \text{in } \R^N,
\end{equation}
where $q < \beta < 2^\ast$. Here the main difficulty is the lack of compactness due to the presence of the critical order term.

First we take $\mu = 0$ and consider the Brezis-Nirenberg type equation
\begin{equation} \label{115}
- \Delta u + (I_\alpha \star |u|^p)\, |u|^{p-2}\, u = \lambda\, |u|^{q-2}\, u + |u|^{2^\ast - 2}\, u \quad \text{in } \R^N,
\end{equation}
where $\lambda > 0$. The next result states that this equation has $m$ distinct pairs of nontrivial solutions for all $\lambda$ in a suitably small left neighborhood of an eigenvalue of problem \eqref{103} with multiplicity $m \ge 1$.

\begin{theorem} \label{Theorem 9}
Let $p < N/(N - 2)$ and let $\lambda_k = \cdots = \lambda_{k+m-1} < \lambda_{k+m}$ for some $k, m \ge 1$. Then $\exists \delta_k > 0$ such that equation \eqref{115} has $m$ distinct pairs of nontrivial solutions at positive energy levels for all $\lambda \in (\lambda_k - \delta_k,\lambda_k)$.
\end{theorem}

In particular, we have the following corollary when $m = 1$.

\begin{corollary}
Let $p < N/(N - 2)$. If $\lambda_k < \lambda_{k+1}$, then $\exists \delta_k > 0$ such that equation \eqref{115} has a pair of nontrivial solutions at a positive energy level for all $\lambda \in (\lambda_k - \delta_k,\lambda_k)$.
\end{corollary}

In general, equation \eqref{114} has arbitrarily many solutions for all sufficiently large $\mu > 0$, depending on $\lambda,$ as it is stated in the following

\begin{theorem} \label{Theorem 10}
Let $p < N/(N - 2)$. For any $\lambda \in \R$ and $m \ge 1$, $\exists \mu_m = \mu_m(\lambda) > 0$ such that equation \eqref{114} has $m$ distinct pairs of nontrivial solutions at positive energy levels for all $\mu > \mu_m$. In particular, the number of solutions goes to infinity as $\mu \to \infty$.
\end{theorem}

In particular, we have the following corollary when $\lambda = 0$.

\begin{corollary}
Let $p < N/(N - 2)$. For any $m \ge 1$, $\exists \mu_m > 0$ such that the equation
\[
- \Delta u + (I_\alpha \star |u|^p)\, |u|^{p-2}\, u = \mu\, |u|^{\beta - 2}\, u + |u|^{2^\ast - 2}\, u \quad \text{in } \R^N,
\]
where $q < \beta < 2^\ast$, has $m$ distinct pairs of nontrivial solutions at positive energy levels for all $\mu > \mu_m$. In particular, the number of solutions goes to infinity as $\mu \to \infty$.
\end{corollary}

\subsection{The case $p>(N+\alpha)/(N-2)$}

Now we consider the case $p>(N+\alpha)/(N-2)$ and assume that $f$ satisfies the Sobolev-supercritical growth condition \eqref{102} with $2^\ast \le q_1 < q_2 < q_{\textrm{rad}}$. Note again that this case occurs only when $N\geq3$. We will consider the subscaled and asymptotically scaled cases separately. We stress that in both cases, because of the different geometry of the relevant energy functionals, the boundedness of PS sequences holds without the assumption $2p<2^*.$

\subsubsection{Subscaled case}

In the subscaled case we have the following existence results.

\begin{theorem} \label{Theorem 11}
Let $p>(N+\alpha)/(N-2)$. Assume that $f$ satisfies \eqref{102} with $q < q_1 < q_2 < q_{\textrm{rad}}$. Then equation \eqref{101} has a solution.
\end{theorem}

We give sufficient conditions for the existence of nontrivial solutions when
\begin{equation} \label{116}
f(|x|,t) = \lambda\, |t|^{q-2}\, t + g(|x|,t),
\end{equation}
where $g$ is a Carath\'{e}odory function on $[0,\infty) \times \R$ satisfying
\begin{equation} \label{117}
|g(|x|,t)| \le a_3\, |t|^{q_3 - 1} + a_4\, |t|^{q_4 - 1} \quad \text{for a.a.\! } x \in \R^N \text{ and all } t \in \R
\end{equation}
for some $a_3, a_4 > 0$ and $2^\ast < q_3 < q_4 < q$.

\begin{theorem} \label{Theorem 12}
Let $p>(N+\alpha)/(N-2)$. Assume that $f$ satisfies \eqref{102} with $q < q_1 < q_2 < q_{\textrm{rad}}$, \eqref{116} and \eqref{117} hold, and that $\lambda$ is not an eigenvalue of problem \eqref{103}.
\begin{enumroman}
\item If $\lambda > \lambda_1$, then equation \eqref{101} has a nontrivial solution.
\item If $\lambda > \lambda_2$, then equation \eqref{101} has two nontrivial solutions.
\end{enumroman}
\end{theorem}

The assumption that $\lambda$ is not an eigenvalue can be removed when the primitive $G(|x|,t) = \int_0^t g(|x|,\tau)\, d\tau$ of $g$ is negative for $t \in \R \setminus \set{0}$.

\begin{theorem} \label{Theorem 13}
Let $p>(N+\alpha)/(N-2)$. Assume that $f$ satisfies \eqref{102} with $q < q_1 < q_2 < q_{\textrm{rad}}$, \eqref{116} and \eqref{117} hold, and that
\begin{equation} \label{118}
G(|x|,t) < 0 \quad \text{for a.a.\! } x \in \R^N \text{ and all } t \in \R \setminus \set{0}.
\end{equation}
\begin{enumroman}
\item If $\lambda > \lambda_1$, then equation \eqref{101} has a nontrivial solution.
\item If $\lambda > \lambda_2$, then equation \eqref{101} has two nontrivial solutions.
\end{enumroman}
\end{theorem}

For example, consider the equation
\begin{equation} \label{119}
- \Delta u + (I_\alpha \star |u|^p)\, |u|^{p-2}\, u = \frac{\lambda\, |u|^{q + \beta - 2}\, u}{1 + |u|^\beta} \quad \text{in } \R^N,
\end{equation}
where $0 < \beta < \min \set{q_{\textrm{rad}} - q,q - 2^\ast}$. The nonlinearity $f(t) = \lambda\, |t|^{q + \beta - 2}\, t/(1 + |t|^\beta)$ satisfies
\[
|f(t)| \le |\lambda|\, |t|^{q + \beta - 1} \quad \forall t \in \R
\]
and $g(t) = f(t) - \lambda\, |t|^{q-2}\, t = - \lambda\, |t|^{q - 2}\, t/(1 + |t|^\beta)$ satisfies
\[
|g(t)| \le |\lambda|\, |t|^{q - \beta - 1} \quad \forall t \in \R.
\]
So the following corollary is immediate from Theorem \ref{Theorem 13}.

\begin{corollary}
Let $p>(N+\alpha)/(N-2)$.
\begin{enumroman}
\item If $\lambda > \lambda_1$, then equation \eqref{119} has a nontrivial solution.
\item If $\lambda > \lambda_2$, then equation \eqref{119} has two nontrivial solutions.
\end{enumroman}
\end{corollary}

\subsubsection{Asymptotically scaled case}

In the asymptotically scaled case we consider the equation
\begin{equation} \label{120}
- \Delta u + (I_\alpha \star |u|^p)\, |u|^{p-2}\, u = \lambda\, |u|^{q-2}\, u + g(|x|,u) \quad \text{in } \R^N,
\end{equation}
where $g$ is a Carath\'{e}odory function on $[0,\infty) \times \R$ satisfying
\begin{equation} \label{121}
|g(|x|,t)| \le a_3\, |t|^{q_3 - 1} + h(x) \quad \text{for a.a.\! } x \in \R^N \text{ and all } t \in \R
\end{equation}
for some $a_3 > 0$, $q < q_3 < q_{\textrm{rad}}$, and $h \in L^\infty(\R^N) \cap L^r(\R^N)$ with $r \in (q_{\textrm{rad}}/(q_{\textrm{rad}} - 1),2^\ast/(2^\ast - 1)]$. We have the following existence result.

\begin{theorem} \label{Theorem 14}
Let $p>(N+\alpha)/(N-2)$. Assume that \eqref{121} holds and that $\lambda$ is not an eigenvalue of problem \eqref{103}. Then equation \eqref{120} has a solution.
\end{theorem}

In particular, we have the following corollary.

\begin{corollary}
Let $p>(N+\alpha)/(N-2)$. Then either the equation
\[
- \Delta u + (I_\alpha \star |u|^p)\, |u|^{p-2}\, u = \lambda\, |u|^{q-2}\, u \quad \text{in } \R^N
\]
has a nontrivial solution, or the equation
\[
- \Delta u + (I_\alpha \star |u|^p)\, |u|^{p-2}\, u = \lambda\, |u|^{q-2}\, u + h(x) \quad \text{in } \R^N
\]
has a solution for all $h \in L^\infty(\R^N) \cap L^r(\R^N)$ with $r \in (q_{\textrm{rad}}/(q_{\textrm{rad}} - 1),2^\ast/(2^\ast - 1)]$.
\end{corollary}

\section{A glance at critical point theorems by scaling} \label{Section 3}

This section is devoted to the relevant abstract tools in critical point theory which will be used to prove our main existence and multiplicity results. Some of the material has been developed recently in \cite{MePe2} with an approach based on the novel notion of {\it scaling in a Banach space} which is defined as follows.

\subsection{Scaling and scaled operators on Banach spaces}

\begin{definition}
    Let $W$ be a reflexive Banach space. A scaling on $W$ is a continuous mapping $W\times[0,\infty)\to W$, $(u,t)\mapsto u_t$ satisfying
    \begin{enumerate}[label=$(H_{\arabic*})$]
    \item $(u_{t_1})_{t_2}=u_{t_1t_2}$ with $t_1,t_2\geq0$ for all $u\in W$ and $t_1,t_2\geq0,$
    \item $(\tau u)_t=\tau u_t$, $\tau\in\R$ $t\geq0$ for all $u\in W$, $\tau\in\R$ and $t\geq0$,
    \item $u_0=0$ and $u_1=u$ for all $u\in W$,
    \item $u_t$ is bounded on bounded sets of $W\times[0,\infty)$.
    \item $\exists s>0$ such that $\|u_t\|=O(t^{s})$ as $t\to\infty$ uniformly on bounded sets.
\end{enumerate}
\end{definition}
 Denote by $W^*$ the dual of $W$. Recall that $q\in C(W,W^*)$ is a potential operator if there is a functional $Q\in C^1(W,\R)$, called a potential for $q$, such that $Q^\prime=q$. By replacing $Q$ with $Q-Q(0)$ if necessary, we may assume that $Q(0)=0$.

 \begin{definition}
     A scaled operator is an odd potential operator $q\in C(W,W^*)$ that maps bounded sets into bounded sets and satisfies
\[q(u_t)v_t=t^sq(u)v,\quad\forall u,v\in W, t\geq0.
     \]
     We denote by $\A_s$ the class of odd scaled potential operators.
     \end{definition}
Let us denote by $I_s$ the potential of the scaled operator $A_s$ with $I_s(0)=0$. Of course $I_s$ is even, bounded on bounded sets, and satisfies the scaling property
\[I_s(u_t)=t^sI_s(u),\quad\forall u\in W, t\geq0\]
(see \cite[Proposition 2.2]{MePe2}).

We will consider the question of existence and multiplicity of solutions to equations of the form
\[
A_s(u)=f(u)\quad\text{in } W^*,
\]
where $f\in C(W,W^*)$ is a potential operator, and $A_s$ is a scaled operator satisfying
\begin{enumerate}[label=$(H_{\arabic*})$, start=6]
\item $A_s(u)u>0$ for all $u\in W\backslash\set{0}$,
\item every sequence $(u_j)$ in $W$ such that $u_j\weak u$ and $A_s(u_j)(u_j-u)\to0$ has a subsequence that converges strongly to $u$.
\end{enumerate}
 Solutions of the above equation coincide with critical points of the $C^1$-functional
\[
\Phi(u)=I_s(u)-F(u),\quad u\in W,
\]
where $F$ is the potential of $f$ with $F(0)=0$.

\subsection{Scaled eigenvalue problems}
Now let us consider the eigenvalue problem
 \begin{equation}\label{3-1}
     A_s(u)=\lambda B_s(u)\quad\text{in } W^*,
 \end{equation}
where $A_s$ and $B_s$ are scaled operators satisfying $(H_6)$ and $(H_7)$ and
\begin{enumerate}[label=$(H_{\arabic*})$, start=8]
    \item $B_s(u)u>0$ for all $u\in W\backslash\set{0}$,
    \item if $u_j\weak u$ in $W$, then $B_s(u_j)\to B_s(u)$ in $W^*$,
\end{enumerate}
and $\lambda\in\R$. We say that $\lambda$ is an eigenvalue if there is a $u\in W\backslash\set{0}$, called an eigenfunction associated with $\lambda$, satisfying equation $\eqref{3-1}$. If that is the case, then $u_t$ is also an aigenfunction associated with $\lambda$ for any $t>0$ since
\[
A_s(u_t)v=A_s(u_t)(v_{1/t})_t=t^sA_s(u)v_{1/t}=t^s\lambda B_s(u)v_{1/t}=\lambda B_s(u_t)(v_{1/t})_t=\lambda B(u_t)v
\]
for all $v\in W$. We denote by $\sigma(A_s,B_s)$ the spectrum, i.e., the set of all eigenvalues, of the pair of scaled operators $(A_s,B_s)$. We have $\sigma(A_s,B_s)\subset (0,\infty)$ by $(H_6)$ and $(H_8)$.

Let us denote by $J_s$ the potential of $B_s$ with $J_s(0)=0$. We also assume that $I_s$ and $J_s$ satisfy
\begin{enumerate}[label=$(H_{\arabic*})$, start=10]
    \item $I_s$ is coercive, i.e., $I_s(u)\to\infty$ as $\|u\|\to\infty$,
    \item the equation $I_s(tu)=1$ has a unique positive solution $t$ for each $u\in W\backslash\set{0}$,
    \item every solution of $\eqref{3-1}$ satisfies
    \[
    I_s(u)=\lambda J_s(u).
    \]
\end{enumerate}

The eigenvalue problem $\eqref{3-1}$ has the following variational formulation. Let
\[
\Psi(u)=\frac{1}{J_s(u)},\quad u\in W\backslash\set{0},
\]
let
\[
\M_s=\set{u\in W:I_s(u)=1},
\]
and let $\w{\Psi}=\Psi|_{\M_s}$. Then $\M_s$ is a complete, symmetric, and bounded $C^1$-Finsler manifold, and eigenvalues of problem $\eqref{3-1}$ coincide with critical values of $\w{\Psi}$ (see \cite[Proposition 2.5]{MePe2}).

\subsubsection{Minimax eigenvalues}
\begin{definition}[Fadell and Rabinowitz \cite{MR0478189}] \label{Definition 301}
Let $\A$ denote the class of symmetric subsets of $W \setminus \set{0}$. For $A \in \A$, let $\overline{A} = A/\Z_2$ be the quotient space of $A$ with each $u$ and $-u$ identified, let $f : \overline{A} \to \RP^\infty$ be the classifying map of $\overline{A}$, and let $f^\ast : H^\ast(\RP^\infty) \to H^\ast(\overline{A})$ be the induced homomorphism of the Alexander-Spanier cohomology rings. The cohomological index of $A$ is defined by
\[
i(A) = \begin{cases}
0 & \text{if } A = \emptyset\\[5pt]
\sup \set{m \ge 1 : f^\ast(\omega^{m-1}) \ne 0} & \text{if } A \ne \emptyset,
\end{cases}
\]
where $\omega \in H^1(\RP^\infty)$ is the generator of the polynomial ring $H^\ast(\RP^\infty) = \Z_2[\omega]$.
\end{definition}

\begin{example}
The classifying map of the unit sphere $S^N$ in $\R^{N+1},\, N \ge 0$ is the inclusion $\RP^N \incl \RP^\infty$, which induces isomorphisms on the cohomology groups $H^l$ for $l \le N$, so $i(S^N) = N + 1$.
\end{example}

The following proposition summarizes the basic properties of the cohomological index.

\begin{proposition}[Fadell and Rabinowitz \cite{MR0478189}] \label{Proposition 300}
The index $i : \A \to \N \cup \set{0,\infty}$ has the following properties:
\begin{enumerate}
\item[$(i_1)$]Definiteness: $i(A) = 0$ if and only if $A = \emptyset$.
\item[$(i_2)$] Monotonicity: If there is an odd continuous map from $A$ to $B$ (in particular, if $A \subset B$), then $i(A) \le i(B)$. Thus, equality holds when the map is an odd homeomorphism.
\item[$(i_3)$] Dimension: $i(A) \le \dim W$.
\item[$(i_4)$] Continuity: If $A$ is closed, then there is a closed neighborhood $N \in \A$ of $A$ such that $i(N) = i(A)$. When $A$ is compact, $N$ may be chosen to be a $\delta$-neighborhood $N_\delta(A) = \set{u \in W : \dist{u}{A} \le \delta}$.
\item[$(i_5)$] Subadditivity: If $A$ and $B$ are closed, then $i(A \cup B) \le i(A) + i(B)$.
\item[$(i_6)$] Stability: If $\Sigma A$ is the suspension of $A \ne \emptyset$, obtained as the quotient space of $A \times [-1,1]$ with $A \times \set{1}$ and $A \times \set{-1}$ collapsed to different points, then $i(\Sigma A) = i(A) + 1$.
\item[$(i_7)$] Piercing property: If $C$, $C_0$, and $C_1$ are closed and $\varphi : C \times [0,1] \to C_0 \cup C_1$ is a continuous map such that $\varphi(-u,t) = - \varphi(u,t)$ for all $(u,t) \in C \times [0,1]$, $\varphi(C \times [0,1])$ is closed, $\varphi(C \times \set{0}) \subset C_0$, and $\varphi(C \times \set{1}) \subset C_1$, then $i(\varphi(C \times [0,1]) \cap C_0 \cap C_1) \ge i(C)$.
\item[$(i_8)$] Neighborhood of zero: If $U$ is a bounded closed symmetric neighborhood of $0$, then $i(\bdry{U}) = \dim W$.
\end{enumerate}
\end{proposition}

Let $\F$ denote the class of symmetric subsets of $\M_s$. For $k\geq1$, let
\[
\F_k=\set{M\in\F:i(M)\geq k}
\]
and set
\[
\lambda_k:=\inf_{M\in\F_k}\sup_{u\in M}\w{\Psi}(u).
\]
We have the following theorem (see Perera et al. \cite[ Proposition 3.52 and Proposition 3.53]{MR2640827})
\begin{theorem}\label{Theorem 301}
    Assume $(H_1)-(H_{12})$. Then $\lambda_k\nearrow \infty$ is a sequence of eigenvalues of $\eqref{3-1}$.
    \begin{enumroman}
        \item\label{301-1} The first eigenvalue is given by
        \[
        \lambda_1=\min_{u\in\M_s}\w{\Psi}(u)>0.
        \]
        \item\label{301-2} If $\lambda_k=\dotsb\lambda_{k+m-1}=\lambda$, and $E_\lambda$ is the set of eigenfunctions associated with $\lambda$ that lie on $\M$, then
        \[i(E_\lambda)\geq m.\]
        \item\label{301-3} If $\lambda_k<\lambda<\lambda_{k+1}$, then
        \[i(\w{\Psi}^{\lambda_k})=i(\M_s\backslash\w{\Psi}_\lambda)=i(\w{\Psi}^\lambda)=i(\M_s\backslash\w{\Psi}_{\lambda_{k+1}})=k,\]
        where $\w{\Psi}^a=\set{u\in\M_s:\w{\Psi}(u)\leq a}$ and $\w{\Psi}_a=\set{u\in\M_s:\w{\Psi}(u)\geq a}$ for $a\in\R$.
    \end{enumroman}
\end{theorem}



\subsection{Subscaled equations}

Consider the equation
\begin{equation}\label{3-2}
    A_s(u)=f(u)
\end{equation}
in $W^*$, where $A_s\in\A_s$ satisfies $(H_6)$, $(H_7)$, and $(H_{10})$, and $f\in C(W,W^*)$ is a compact potential operator satisfying
\begin{equation}\label{3-3}
    f(u_t)v_t=o(t^s)\|v\|\quad\text{as }t\to\infty
\end{equation}
uniformly in $u$ on bounded sets for all $v\in W.$ This equation is referred as subscaled. The variational functional associated with equation $\eqref{3-2}$ is
\[\Phi(u)=I_s(u)-F(u),\quad u\in W,
\]
where
\[I_s(u)=\int_0^1A_s(\tau u)ud\tau,\quad F(u)=\int_0^1 f(\tau u)ud\tau
\]
are the potentials of $A_s$ and $f$ respectively, with $I_s(0)=0=F(0)$. We have the following compactness fact, \cite[Proposition 2.3]{MePe2}.

\begin{proposition}\label{PS compact}
    If $f$ is a compact operator, then every bounded sequence $(u_j)$ in $W$ such that $\Phi^\prime(u_j)\to0$ has a convergent subsequence.
\end{proposition}

\begin{remark}\label{remark1}
    Since the compact operator $f$ maps bounded sets into precompact, and hence bounded, sets, $F$ is bounded on bounded sets. Since $I_s$ is also bounded on bounded sets, then so is $\Phi$. It is not hard to see that $\Phi$ is coercive (see \cite{MePe2}, Lemma 2.15). It follows that $\Phi$ satisfies the $(PS)$ condition By Proposition \ref{PS compact} and therefore the number $\min\Phi(W)$ is a critical value for $\Phi$
\end{remark}
 In summary we have (\cite[Theorem 2.16]{MePe2}).

\begin{theorem}\label{Theorem 302}
    Assume $(H_1)$-$(H_7)$, $(H_{10})$, and $\eqref{3-3}$. Then equation $\eqref{3-2}$ has a solution.
\end{theorem}

Now we assume that
\begin{equation}\label{3-4}
    f(u)=\lambda B_s(u)+g(u),
\end{equation}
where $B_s\in\A_s$ satisfies $(H_8)$, $(H_9)$ and $(H_{12})$, $g\in C(W,W^*)$ is a compact potential operator satisfying
\begin{equation}\label{3-5}
    g(u_t)v_t=o(t^s)\|v\|\quad\text{as }t\to0
\end{equation}
uniformly in $u$ on bounded sets for all $v\in W$, and $\lambda\in\R$. It is not hard to see that in this case $u=0$ is a solution of equation $\eqref{3-4}$. Regarding nontrivial solutions to $\eqref{3-4}$ we have (\cite[Theorem 2.17]{MePe2})
\begin{theorem}\label{Theorem 303}
    Assume $(H_1)-(H_{12})$, \eqref{3-3}, \eqref{3-4}, and \eqref{3-5}, and let $\lambda\in\R\backslash\sigma(A_s,B_s)$.
    \begin{enumroman}
        \item If $\lambda>\lambda_1$, then equation \eqref{3-2} has a nontrivial solution.
        \item If $\lambda>\lambda_2$, then equation \eqref{3-2} has two nontrivial solutions.
    \end{enumroman}
\end{theorem}
\subsection{Scaled linking theorem}
\begin{theorem}[Mercuri \& Perera \cite{MePe2}]\label{scaled linking theorem}
    Let $\Phi$ be a $C^1$-functional on $W$. Let $A_0$ and $B_0$ be disjoint nonempty closed symmetric subsets of $\M$ such that
    \begin{eqnarray*}
        i(A_0)=i(\M\backslash B_0)<\infty,
    \end{eqnarray*}
    let $R>\rho>0$, let $e\in\M\backslash A_0$, and let
    \begin{eqnarray*}
        X&=&\set{(\pi((1-\tau)u+\tau e))_t:u\in A_0,\tau\in[0,1], 0\leq t\leq R},\\
        A&=&\set{u_t:u\in A_0, 0\leq r\leq R}\cup\set{(\pi((1-\tau)u+\tau e))_R:u\in A_0, \tau\in[0,1]},\\
        B&=&\set{u_\rho:u\in B_0}.
    \end{eqnarray*}
    Assume that
    \begin{eqnarray}\label{scalin1}
        \sup_{u\in A}\Phi(u)\leq\inf_{u\in B}\Phi(u),\quad\sup_{u\in X}\Phi(u)<\infty.
    \end{eqnarray}
    Let 
    \begin{eqnarray*}
        \Gamma=\set{\gamma\in C(X,W):\gamma(X)\text{ is closed and }\gamma|_{A}=id}
    \end{eqnarray*}
    and set 
    \begin{eqnarray}\label{scalin2}
        c:=\inf_{\gamma\in\Gamma}\sup_{u\in\gamma(X)}\Phi(u).
    \end{eqnarray}
    Then
    \begin{eqnarray*}
        \inf_{u\in B}\Phi(u)\leq c\leq\sup_{u\in X}\Phi(u).
    \end{eqnarray*}
    If $\Phi$ satisfies the $(PS)_c$ condition, then $c$ is a critical value of $\Phi$.
\end{theorem}
\subsection{Asymptotically scaled equations}
 Now we consider equations of the type
 \begin{equation}\label{3-6}
     A_s(u)=\lambda B_s(u)+g(u)
 \end{equation}
 in $W^*$, where $A_s,B_s\in\A_s$ satisfy $(H_6)$-$(H_{12})$, $g\in C(W,W^*)$ is a compact potential operator satisfying
 \begin{equation}\label{3-7}
     g(u_t)v_t=o(t^s)\norm{v}\quad\text{as }t\to\infty
 \end{equation}
 uniformly in $u$ on bounded sets for all $v\in W$, and $\lambda\in\R$. This equation is called an asymptotically scaled equation. 
 
 The variational functional associated with equation \eqref{3-6} is given by
 \[
 \Phi(u)=I_s(u)-\lambda J_s(u)-G(u),\quad u\in W,
 \]
 We have the following theorem \cite[Theorem 2.25]{MePe2}.
 \begin{theorem}\label{Theorem 304}
     Assume $(H_1)$-$(H_{12})$ and $\eqref{3-7}$, and let $\lambda\in\R\backslash\sigma(A_s,B_s)$. Then equation \eqref{3-6} has a solution.
 \end{theorem}

 An important step in the proof of Theorem \ref{Theorem 304} presented in \cite{MePe2} is the following.

 \begin{proposition}\label{Proposition 304-1}
     Assume $(H_1)$-$(H_{12})$ and $\eqref{3-7}$, and let $\lambda\in\R\backslash\sigma(A_s,B_s)$, then $\Phi$ satisfies the $(PS)$ condition.
 \end{proposition}

 \subsection{Superscaled equations}
Consider the equation
\begin{equation}\label{3-8}
A_s(u)=f(u)
    \end{equation}
in $W^*$, where $A_s\in\A_s$ satisfies $(H_6)$, $(H_7)$ and $(H_{10})$, and $f\in C(W,W^*)$ is a potential operator with potential
\[
F(u)=\int_0^1f(\tau u)ud\tau
\]
satisfying 
\begin{equation}\label{3-9}
    \frac{F(u_t)}{t^s}\to\infty\quad\text{as }t\to\infty
\end{equation}
uniformly in $u$ on compact subsets of $\M_s$.

We assume that
\begin{equation}\label{3-10}
    f(u)=\lambda B_s(u)+g(u)
\end{equation}
where $B_s\in\A_s$ satisfies $(H_8)$, $(H_9)$, and $(H_{12})$, $g\in C(W,W^*)$ is a potential operator satisfying
\begin{equation}\label{3-11}
    g(u_t)v_t=o(t^s)\|v\|\quad\text{as }t\to0
\end{equation}
uniformly in $u$ on bounded sets for all $v\in W$, and $\lambda\in\R\backslash\sigma(A_s,B_s)$. 

The variation functional associated with \eqref{3-8} is given by
\[
\Phi(u)=I_s(u)-\lambda J_s(u)-G(u),\quad u\in W,
\]
where $I_s$ and $J_s$ are the potentials of $A_s$ and $B_s$ respectively, with $I_s(0)=0=J_s(0)$, and 
\[
G(u)=\int_0^1g(\tau u)ud\tau
\]
is the potential of $g$ with $G(0)=0$. We further assume that
\begin{equation}\label{3-12}
    G(u)\geq0\quad\forall u\in W.
\end{equation}

In this configuration we have the following result \cite[Theorem 2.27]{MePe2}.

\begin{theorem}\label{Theorem 406}
    Assume $(H_1)$-$(H_{12})$, \eqref{3-9},\eqref{3-10}, \eqref{3-11}, and \eqref{3-12}, let $\lambda\in\R\backslash\sigma(A_s,B_s)$. If $\Phi$ satisfies the $(PS)$ condition, then equation \eqref{3-8} has a nontrivial solution at a positive level.
\end{theorem}

 \subsection{Local linking based on scaling}
\begin{definition}\label{local linking}
    We will say that $\Phi$ has a scaled local linking near the origin in dimension $k\geq1$ if there are disjoint nonempty closed symmetric subsets $A_0$ and $B_0$ of $\M$ with
    \begin{eqnarray*}
        i(A_0)=i(\M\backslash B_0)=k
    \end{eqnarray*}
    and $\rho>0$ such that
    \begin{eqnarray*}
        \begin{cases}
            \Phi(u_t)\leq0,\quad\forall u\in A_0, 0\leq t\leq\rho\\
            \Phi(u_t)>0,\quad\forall u\in B_0, 0<t\leq\rho.
        \end{cases}
    \end{eqnarray*}
\end{definition}
To understand the usefulness of this notion, let us recall the definition of critical groups.
\begin{definition}\label{criticalgroups}
    Let $\Phi$ be a $C^1$-functional on a Banach space $W$ and let $u_0$ be an isolated critical point of $\Phi$. The critical groups of $E$ in $u_0$ are defined by
    \[
    C^q(\Phi,u_0)=H^q(\Phi^c\cap U,\Phi^c\cap U\backslash\set{u_0}),\quad q\ge0,
    \]
    where $c=\Phi(u_0)$ is the corresponding critical value and $U$ is a neighborhood of $u_0$ that does not contain critical points other than $u_0$.
\end{definition}
We have the following theorem (see \cite[Theorem 2.29]{MePe2})
\begin{theorem}\label{Theorem 305}
    If $\Phi$ has a scaled local linking near the origin in dimension $k$, then
    \begin{eqnarray*}
        C^k(\Phi,0)\neq0.
    \end{eqnarray*}
\end{theorem}

We also have the following result \cite[Theorem 2.30]{MePe2}.
\begin{theorem}\label{Theorem 306}
    Let $\Phi$ be a $C^1$-functional on $W$ that satisfies the $(PS)$ condition. Assume that $\Phi$ has a scaled local linking near the origin in dimension $k$ and that $\w{H}^{k-1}(\Phi^a)=0$ for some $a<0$. Then $\Phi$ has a nontrivial critical point.
\end{theorem}
\begin{corollary}\label{Corollary 307}
    Let $\Phi$ be a $C^1$-functional on $W$ that satisfies the $(PS)$ condition. Assume that $\Phi$ has a scaled local linking near the origin and that $\Phi^a$ is contractible for some $a<0$. Then $\Phi$ has a nontrivial critical point.
\end{corollary}

\subsection{Multiplicity based on scaling}

Let $\Phi\in C^1(W,\R)$ be an even functional, i.e., $\Phi(-u)=\Phi(u)$ for all $u\in W$. Assume that $\exists c^*>0$ such that $\Phi$ satisfies the $(PS)_c$ condition for all $c\in(0,c^*)$. Let $\Gamma$ denote the group of odd homeomorphisms of $W$ that are the identity outside $\Phi^{-1}(0,c^*)$. Let $\A^*$ denote the class of symmetric subsets of $W$, and let
\begin{eqnarray*}
    \M_\rho=\set{u\in W: I_s(u)=\rho^s}=\set{u_\rho:u\in\M}
\end{eqnarray*}
for $\rho>0.$
\begin{definition}[Benci \cite{MR84c:58014}]
    The pseudo-index of $M\in\A^*$ related to $i$, $\M_\rho$, and $\Gamma$ is defined by 
    \begin{eqnarray*}
        i^*(M)=\min_{\gamma\in\Gamma} i(\gamma(M)\cap\M_\rho).
    \end{eqnarray*}
\end{definition}
We have the following multiplicity result.

\begin{theorem}[Mecuri \& Perera \cite{MePe2}]\label{m1}
    Let $A_0$ and $B_0$ be symmetric subsets of $\M$ such that $A_0$ is compact, $B_0$ is closed, and
    \begin{eqnarray}\label{m2}
        i(A_0)\geq k+m-1,\quad i(\M\backslash B_0)\leq k-1
    \end{eqnarray}
    for some $k,m\geq1$. Let $R>\rho>0$ and let
    \begin{eqnarray*}
        X&=&\set{u_t:u\in A_0, 0\leq t\leq R},\\
        A&=&\set{u_R:u\in A_0},\\
        B&=&\set{u_{\rho}:u\in B_0}.
        \end{eqnarray*}
        Assume that
        \begin{eqnarray}\label{m1-2}
            \sup_{u\in A}\Phi(u)\leq0<\inf_{u\in B}\Phi(u),\quad \sup_{u\in X}\Phi(u)<c^*.
        \end{eqnarray}
        For $j=k,\dots,k+m-1$, let
        \begin{eqnarray*}
            \A_j^*=\set{M\in\A^*:M\text{ is compact and }i^*(M)\geq j}
        \end{eqnarray*}
        and set
        \begin{eqnarray*}
            c^*_j=\inf_{M\in\A_j^*}\max_{u\in M}\Phi(u).
        \end{eqnarray*}
        Then $0<c_k^*\leq\cdots\leq c_{k+m-1}^*<c^*$, each $c_j^*$ is a critical value of $\Phi$, and $\Phi$ has $m$ distinct pairs of associated critical points.
\end{theorem}

\begin{corollary}\label{m2_corollary}
    Let $A_0$ be a compact symmetric subset of $\M$ with $i(A_0)=m\geq1$, let $R>\rho>0$, and let
    \begin{eqnarray*}
        A=\set{u_R:u\in A_0},\quad X=\set{u_t:u\in A_0,0\leq t\leq R}.
    \end{eqnarray*}
    Assume that
    \begin{eqnarray*}
        \sup_{u\in A}\Phi(u)\leq0<\inf_{u\in\M_\rho}\Phi(u),\quad\sup_{u\in X}\Phi(u)<c^*.
    \end{eqnarray*}
    Then $\Phi$ has $m$ distinct pairs of critical points at levels in $(0,c^*)$.
\end{corollary}

\section{Proof of the main results} \label{Section 4}

In this section we present the proof of our main results, which also rely on some preliminary facts of independent interest on the structure of problem \eqref{101}, which we collect in the following Lemma \ref{lemma1bis} and Lemma \ref{Lemma1}. With the abstract results of the preceding section at hand, we consider now in particular $W=E^{\alpha,\,p}_{\text{rad}}(\R^N)$ with the norm
\[
\norm{u} = \left[\int_{\R^N} |\nabla u|^2\, dx + \left(\int_{\R^N} \abs{I_{\alpha/2} \star |u|^p}^2\, dx\right)^{1/p}\right]^{1/2}.
\]

On $W$ we define a scaling as follows. Let $\sigma = (2 + \alpha)/2(p - 1)$ and consider the map
\[
E^{\alpha,\,p}_{\text{rad}}(\R^N) \times [0,\infty) \to E^{\alpha,\,p}_{\text{rad}}(\R^N), \quad (u,t) \mapsto u_t
\]
given by
\begin{equation}\label{scaling1}
    u_t(x) = t^\sigma u(tx)
\end{equation}
if $1/p> (N - 2)/(N + \alpha)$ and
\begin{equation}\label{scaling2}
    u_t(x) = \begin{cases}
t^{- \sigma} u(x/t) & \text{if } t > 0\\[5pt]
0 & \text{if } t = 0
\end{cases}
\end{equation}
if $1/p< (N - 2)/(N + \alpha)$.

As in \cite{MePe2}, it is easy to see that the mapping $(u,t)\mapsto u_t$ is continuous and satisfies $(H_1)$-$(H_5)$ with
\[
s=\begin{cases}
    \frac{p(2-N)+N+\alpha}{p-1} & \text{if } 1/p> (N - 2)/(N + \alpha),\\[5pt]
    -\frac{p(2-N)+N+\alpha}{p-1} & \text{if }1/p< (N - 2)/(N + \alpha).
\end{cases}
\]

In this setting, the operators $A_2$ and $B_s$ are given by
\[
A_s(u)v=\int_{\R^N}\nabla u\cdot\nabla vdx+\int_{\R^N}(I_{\alpha/2}\star|u|^p)(I_{\alpha/2}\star|u|^{p-2}uv)dx,
\]
and
\[
B_s(u)v=\int_{\R^N}|u|^{q-2}uvdx,
\]
where $q = 2\, (2p + \alpha)/(2 + \alpha)$. Note that $A_s,B_s\in\A_s$, and satisfy $(H_6)$, $(H_8)$ and $(H_9)$. The following lemma guaranties that condition $(H_7)$ is also satisfied.

\begin{lemma}\label{lemma1bis}
    Every sequence $(u_j)$ in $E^{\alpha,\,p}(\R^N)$ weakly converging to $u$ and such that $(A_s(u_j)-A_s(u))(u_j-u)\to0,$ strongly converges to $u$.
\end{lemma}

\begin{proof}
We have
\begin{align*}
      o(1)+  (A_s(u_j)-A_s(u))(u_j-u)&=\int_{\R^N}|\nabla u_j-\nabla u|^2dx+\int_{\R^N} (I_{\alpha/2}\star |u_j|^p)^2dx\\
        &-\int_{\R^N} (I_{\alpha/2}\star |u_j|^p) (I_{\alpha/2}\star |u_j|^{p-2}u_ju)dx\\
        &- \int_{\R^N} (I_{\alpha/2}\star |u|^p)(I_{\alpha/2}\star |u|^{p-2}uu_j)dx\\
        &+\int_{\R^N} (I_{\alpha/2}\star |u|^p)^2dx\\
        &\geq \int_{\R^N}|\nabla u_j-\nabla u|^2dx +\frac{1}{p}\int_{\R^N}\Big(I_{\alpha/2}\star|u_j|^p-I_{\alpha/2}\star|u|^p\Big)^2,
 \end{align*}
    where we have only used Young's inequality to estimate the terms $|u_j|^{p-2}u_ju$ and $|u|^{p-2}uu_j.$ The conclusion follows then by
    uniform convexity \cite[Proposition 2.9]{MR3568051} or by the Brezis-Lieb inequality \cite[Proposition 4.1]{MR3568051}. 
\end{proof}

In this setting we have
\[
I_s(u)=\frac{1}{2}\int_{\R^N}|\nabla u|^2dx+\frac{1}{2p}\int_{\R^N}|I_{\alpha/2}\star|u|^p|^2dx,
\]
\[
\M_s=\set{u\in E^{\alpha,\,p}_{\text{rad}}(\R^N):I_s(u)=1},
\]
\[
J_s(u)=\frac{1}{q}\int_{\R^N}|u|^qdx,
\]
\[
\widetilde{\Psi}(u)=\dfrac{1}{J_s(u)}.
\]
It is easy to see that $(H_{10})$ and $(H_{11})$ hold. Condition $(H_{12})$ can be checked as follows. If $u$ is a solution of $\eqref{103}$, it implies that, after testing against $u,$
\[
\om|\nabla u|^2dx+\col{u}=\lambda\sca{u},
\]
and $u$ satisfies the Pohozaev type identity \cite[Proposition 5.5]{MR3568051},
\[
\frac{N-2}{2}\om|\nabla u|^2dx+\frac{N+\alpha}{2p}\col{u}=\lambda\frac{N}{q}\sca{u},
\]
which implies, putting togeteher the last two equations, that
\[
\frac{1}{2}\om|\nabla u|^2x+\frac{1}{2p}\col{u}=\frac{\lambda}{q}\sca{u}.
\]
This shows that $(H_{12})$ also holds.

Let us set
\[
\w{f}(u)v=\om f(|x|,u)vdx,\quad\text{and}\quad\w{g}(u)v=\om g(|x|,u)vdx.
\]

We will also need the following asymptotic properties.

\begin{lemma}\label{Lemma1}
    We have the following asymptotic estimates on $\w{f}$ and $\w{g}$:
    \begin{enumroman}
        \item\label{Lemma1-1} If $f$ satisfies \eqref{102} with $\p<q_1<q_2<q$ if $1/p> (N - 2)/(N + \alpha)$ or $q<q_1<q_2<\p$ if $1/p< (N - 2)/(N + \alpha)$, then
        \[
        \w{f}(u_t)v_t=o(t^s)\norm{v}\quad\text{as }t\to\infty
        \]
        uniformly in $u$ on bounded sets for all $v\in\e$.
        \item\label{Lemma1-2} If $g$ satisfies $\eqref{105}$ if $1/p> (N - 2)(N + \alpha)$, or \eqref{117} if $1/p< (N - 2)/(N + \alpha)$, then
        \[
        \w{g}(u_t)v_t=o(t^s)\norm{v}\quad\text{as }t\to0
        \]
        uniformly in $u$ on bounded sets for all $v\in\e$.
        \item\label{Lemma1-3} If $g$ satisfies \eqref{109} if $1/p> (N - 2)/(N + \alpha)$, or \eqref{121} if $1/p< (N - 2)/(N + \alpha)$, then
        \[
        \w{g}(u_t)v_t=o(t^s)\norm{v}\quad\text{as }t\to\infty
        \]
        uniformly in $u$ on bounded sets for all $v\in\e$.
    \end{enumroman}
\end{lemma}
\begin{proof}
    We will prove only \ref{Lemma1-1} as the other cases are similar.

    First suppose $1/p> (N - 2)/(N + \alpha)$ and that $f$ satisfies \eqref{102} with $\p<q_1<q_2<q$. In this case the scaling is given by \eqref{scaling1}. Note that
    \[
    \om|a||v_t|dx\leq t^\sigma\seq{\om|a|^rdx}^{1/r}\seq{\om|v(tx)|^{r^\prime}dx}^{1/r^\prime}=t^{\sigma-N/r^\prime}|a|_r|v|_{r^\prime}
    \]
    This together with \eqref{102} gives
   \begin{eqnarray*}
        |\w{f}(u_t)v_t|\leq a_1\om|u_t|^{q_1-1}|v_t|dx+a_2\om|u_t|^{q_2-1}|v_t|dx+\om|a||v_t|dx\\
       \leq a_1t^{\sigma q_1-N}|u|_{q_1}^{q_1-1}|v|_{q_1}+a_2t^{\sigma q_2-N}|u|_{q_2}^{q_2-1}|v|_{q_2}+t^{\sigma-N/r^\prime}|a|_r|v|_{r^\prime}
   \end{eqnarray*}
   where $1/r+1/r^\prime=1$. Now, the condition that $\sigma-N/r^\prime-s<0$ is equivalent to $N/(N-\sigma(p-1))>r^\prime$, which is satisfied when $1/p> (N - 2)/(N + \alpha)$, as $2^*>N/(N-\sigma(p-1))$. Since $q_1,q_2<q$, we have $\sigma q_i-N-s<0$, $i=1,2$, and this completes the proof.

\end{proof}

\subsection{Proof of Theorem \ref{Theorem 2}, Theorem \ref{Theorem 3}, Theorem \ref{Theorem 11} and Theorem \ref{Theorem 12}}

\begin{proof}[Proof of Theorem \ref{Theorem 2} and Theorem \ref{Theorem 11}]
    Since $\eqref{3-3}$ holds by Lemma \ref{Lemma1} \ref{Lemma1-1}, the conclusion follows from Theorem \ref{Theorem 302}.
\end{proof}

\begin{proof}[Proof of Theorem \ref{Theorem 3} and \ref{Theorem 12}]
    Since \eqref{3-3} and \eqref{3-5} hold by Lemma \ref{Lemma1} \ref{Lemma1-1} and \ref{Lemma1-2}, respectively, the conclusion follows from Theorem \ref{Theorem 303}.
\end{proof}

\subsection{Proof of Theorem \ref{Theorem 4} and Theorem \ref{Theorem 13}}

The variational functional associated to equation \eqref{101} is given by
\[
\Phi(u)=I_s(u)-\lambda J_s(u)-\w{G}(u),\quad u\in\e,
\]
where
\[
\w{G}(u)=\om G(|x|,u)dx.
\]

We observe that the following lemma (proved in \cite[Lemma 3.5]{MePe2} in a special case) holds. Since the proof would follow the same lines as in \cite{MePe2}, we leave it out.
\begin{lemma}\label{Lemma2}
    Assume $f$ satisfies \eqref{102} with $\p<q_1<q_2<q$, \eqref{104} and \eqref{105} hold if $1/p> (N - 2)/(N + \alpha)$, and that $f$ satisfies \eqref{102} with $q<q_1<q_2<\p$, \eqref{116} and \eqref{117} hold if $1/p< (N - 2)/(N + \alpha)$. Then $\Phi$ has a scaled local linking near the origin in dimension $k$ in each of the following cases:
    \begin{enumroman}
        \item\label{Lemma2-1} $\lambda_k<\lambda\leq\lambda_{k+1}$ and $G(|x|,t)<0$ for a.a. $x\in\R^N$ and all $t\in\R\backslash\set{0}$,
       \item\label{Lemma2-2} $\lambda_k\leq\lambda<\lambda_{k+1}$ and $G(|x|,t)\geq0$ for a.a. $x\in\R^N$ and all $t\in\R$.
    \end{enumroman}
\end{lemma}

We will also need the following proposition, which was proved in \cite{perera1998critical}.
\begin{proposition}\label{Proposition 4-1}
    Let $\Phi$ be a $C^1$-functional on a Banach space $W$ that is bounded from below and satisfies the $(PS)$ condition. Assume that $0$ is a critical point of $\Phi$ with $\Phi(0)=0$ and $C^k(\Phi,0)\neq0$ for some $k\geq1$. Then $\Phi$ has a critical point $u_1\neq0$ with either $\Phi(u_1)<0$ and $C^{k-1}(\Phi,u_1)\neq0$, or $\Phi(u_1)>0$ and $C^{k+1}(\Phi,u_1)\neq0$.
\end{proposition}
\begin{proof}[Proof of Theorem \ref{Theorem 4} and Theorem \ref{Theorem 13}]
    By Remark \ref{remark1}, it follows that $\Phi$ satisfies the $(PS)$ condition, and it has a minimizer $u_0$. If $u_0$ is not an isolated minimizer, then $\Phi$ must have infinitely many critical points, so we may assume that $u_0$ is isolated. Then the critical groups of $\Phi$ at $u_0$ are given by
    \begin{equation}\label{Th 23}
        C^l(\Phi,u_0)\approx\delta_{l0}\Z_2,
    \end{equation}
    (see \cite[Example 4.1]{chang2012infinite}).
    
    \ref{Theorem 4-1}. If $\lambda>\lambda_1$, then $\lambda_k<\lambda\leq\lambda_{k+1}$ for some $k\geq1$. By Lemma \ref{Lemma2} \ref{Lemma2-1}, $\Phi$ has a scaled local linking near the origin in dimension $k$. Then $C^k(\Phi,0)\neq0$ by Theorem \ref{Theorem 305}. Since $C^k(\Phi,u_0)=0$ by \eqref{Th 23}, it follows that $u_0\neq0$.

    \ref{Theorem 4-2}. If $\lambda>\lambda_2$, then by the same reason as above, we have $C^k(\Phi,0)\neq0$ for some $k\geq2$. Proposition \ref{Proposition 4-1} gives a critical point $u_1\neq0$ with either $\Phi(u_1)<0$ and $C^{k-1}(\Phi,u_1)\neq0$, or $\Phi(u_1)>0$ and $C^{k+1}(\Phi,u_1)\neq0$. Since $k\geq2$, $C^{k-1}(\Phi,u_0)=C^{k+1}(\Phi,u_0)=0$ by \eqref{Th 23}, so $u_1\neq u_0$.
\end{proof}

\subsection{Proof of Theorem \ref{Theorem 5} and Theorem \ref{Theorem 14}}
\begin{proof}[Proof of Theorem \ref{Theorem 5} and Theorem \ref{Theorem 14}]
    Since \eqref{3-7} holds from Lemma \ref{Lemma1} \ref{Lemma1-3}, the desired conclusion follows from Theorem \ref{Theorem 304}.
\end{proof}

\subsection{Proof of Theorem \ref{Theorem 6} and Theorem \ref{Theorem 8}} 
The variational functional associated with equation \eqref{110} is given by
\[
\Phi(u)=I_s(u)-\lambda J_s(u)-\w{G}(u)-\frac{1}{r}\om|u|^rdx
\]
where
\[
\w{G}(u)=\om G(|x|,u)dx.
\]
 First we prove that $\Phi$ satisfies the $(PS)$ condition.

 \begin{lemma}\label{Lemma3}
     Suppose $1/p>(N-2)/N$. If \eqref{111} holds, then $\Phi$ satisfies the $(PS)$ condition.
 \end{lemma}
 \begin{proof}
       First we prove that the sequence is bounded. Suppose that $\|u_j\|\to\infty$. Define
  \begin{eqnarray*}
      t_j=\dfrac{1}{I_s(u_j)^{1/s}},\quad \w{u}_j=(u_j)_{t_j},\quad\w{t_j}=\dfrac{1}{t_j}=I_s(u_j)^{1/s},
  \end{eqnarray*}
  then $u_j=(\w{u}_j)_{\w{t_j}}$. We have $\w{t_j}\to\infty$ since $\|u_j\|\to\infty$. Then $\Psi(u_j)=c+o(1)$ and $\Psi^\prime(u_j)u_j$ become
  \begin{eqnarray}\label{2PS-1}
      \w{t_j}^s\seq{\frac{1}{2}\dev{\w{u}_j}+\frac{1}{2p}\col{\w{u}_j}-\frac{\lambda}{q}\sca{\w{u}_j}}\nonumber\\
      =\int_{\R^N}G(|x|,u_j)dx+\frac{1}{r}\om|u_j|^rdx+c+o(1)
  \end{eqnarray}
  and
  \begin{eqnarray}\label{2PS-2}
      \w{t_j}^s\seq{\dev{\w{u}_j}+\col{\w{u}_j}-\sca{\w{u}_j}}\nonumber\\
     = \int_{\R^N}g(|x|,u_j)u_jdx+\om|u_j|^rdx+o(\w{t_j}^{s/2})
  \end{eqnarray}
  respectively. We have
  \begin{eqnarray*}
      \om|u_j|^rdx=\w{t_j}^{r\sigma-N}\om|\w{u}_j|^rdx.
  \end{eqnarray*}

 Also,
 \begin{eqnarray}\label{nPS}
     \w{t_j}^s\seq{\frac{1}{2}\dev{\w{u}_j}+\frac{1}{2p}\col{\w{u}_j}-\frac{\lambda}{q}\sca{\w{u}_j}}\nonumber\\
      =\frac{1}{r}\om|u_j|^rdx-\frac{a_3}{q_3}\om|u_j|^{q_3}dx\nonumber\\-\frac{a_4}{q_4}\om|u_j|^{q_4}dx+c+o(1)
 \end{eqnarray}
  
  by \eqref{2PS-1} and \eqref{111}. 

  Since $q<q_3<q_4<r$, we have the interpolation inequalities
  \[
  \om|u|^{q_i}dx\leq\seq{\om|u|^qdx}^{1-\theta_i}\seq{\om|u|^rdx}^{\theta_i},\quad 
  \]
  where $i=3,4$, and $\theta_i=(q_i-q)/(r-q)$. This implies that
  \[
  t^{\sigma(q_i-q)}\om|\w{u}_j|^{q_i}dx\leq c_i\seq{t^{\sigma(r-q)}\om|\w{u}_j|^rdx}^{\theta_i}
  \]
  since $\w{u}_j$ is a bounded sequence in $L^q(\R^N)$. This together with $\eqref{2PS-1}$ gives
  \begin{eqnarray*}
       \frac{1}{2}\dev{\w{u}_j}+\frac{1}{2p}\col{\w{u}_j}\geq\frac{\w{t}_j^{\sigma(r-q)}}{r}\om|\w{u}_j|^{r}dx\\
       -\w{c}_1\seq{\w{t}_j^{\sigma(r-q)}\om|\w{u}_j|^rdx}^{\theta_1}-\w{c}_2\seq{\w{t}_j^{\sigma(r-q)}\om|\w{u}_j|^rdx}^{\theta_2}
       \end{eqnarray*}
       for some constants $\w{c}_i>0$, $i=1,2.$.
Since the left hand side of the above inequality is bounded, this implies that
\begin{eqnarray}\label{2PS-3}
    \om|\w{u}_j|^rdx=O(\w{t}_j^{-\sigma(r-q)}).
\end{eqnarray}
 Since $(\w{u}_j)$ is a bounded sequence in $L^l(\R^N)$ for all $l\in(\p,q),$ the interpolation inequality
 \[
 \om|u|^{q_3}dx\leq|u|_l^{l(1-\theta)}|u|_r^{r\theta},
 \]
 where $q_3=(1-\theta)l+\theta r$, implies
 \[
 \w{t}_j^{\sigma(q_3-q)}\om|\w{u}_j|^{q_3}dx=O\seq{\w{t}_j^{-\sigma(r-q)\frac{q_3-l}{r-l}+\sigma(q_3-q)}},
 \]
 and this implies that $\w{t}_j^{\sigma(q_3-q)}\om|\w{u}_j|^{q_3}dx\to0$ as $j\to\infty$ since $-\sigma(r-q)\frac{q_3-l}{r-l}+\sigma(q_3-q)<0$. For the same reason, $\w{t}_j^{\sigma(q_4-q)}\om|\w{u}_j|^{q_3}dx\to0$ as $j\to\infty$. Note that \eqref{2PS-3} also implies, after an interpolation argument, that $\sca{\w{u}_j}\to0$. Using those facts and \eqref{111}, it follows, after dividing \eqref{2PS-2} by $r$ and subtracting it from \eqref{2PS-1} that 
 \[
 \seq{\frac{1}{2}-\frac{1}{r}}\dev{\w{u}_j}+\seq{\frac{1}{2p}-\frac{1}{r}}\col{\w{u}_j}\leq o(1)
 \]
 which is a contradiction, because $\w{u}_j\in\M_s$ for all $j$. The desired conclusion now follows from Proposition \ref{PS compact}.
\end{proof}

\begin{proof}[Proof of Theorem \ref{Theorem 6}]
The variational functional associated with equation \eqref{110} is given by
\[
\Phi(u)=I_s(u)-\lambda J_s(u)-\om G(|x|,u)dx-\frac{1}{r}\om|u|^rdx.
\]
In particular,
\begin{eqnarray}\label{27-1}
    \Phi(u_t)=t^s\seq{1-\frac{\lambda}{\w{\Psi}(u)}}-\om G(|x|,u_t)dx-\frac{t^{\sigma r-N}}{r}\om|u|^rdx\quad\forall u\in\M_s,
\end{eqnarray}
consequently
\begin{eqnarray}\label{27-2}
    \Phi(u_t)=t^s\seq{1-\frac{\lambda}{\w{\Psi}(u)}+o(1)}\quad\text{as }t\to0
\end{eqnarray}
uniformly on $\M_s$ by \eqref{111}.

   If $\lambda<\lambda_1$ then
   \[
   \Phi(u_t)\geq t^s\seq{1-\frac{\lambda^+}{\lambda_1}+o(1)}\quad\text{as }t\to0
   \]
   uniformly on $\M_s$, where $\lambda^+=\max\set{\lambda,0}$, so $\Phi$ has the mountain pass geometry, and consequently equation \eqref{110} has a solution. Now, assume that $\lambda_k<\lambda<\lambda_{k+1}$ for some $k\geq1$. Since $\M_s\backslash\w{\Psi}_\lambda$ is an open symmetric set of index $k$ by Theorem \ref{Theorem 301} \ref{301-3}, then it has a compact symmetric subset $A_0$ of idex $k$ (see the proof of Proposition 3.1 in Degiovanni and Lancelotti \cite{MR2371112}). Let $B_0=\w{\Psi}_{\lambda_{k+1}}$, hence $i(\M_s\backslash B_0)=k$. 

   As in Theorem \ref{scaled linking theorem}, let $R>\rho>0$, let $e\in\M\backslash A_0$, and let
    \begin{eqnarray*}
        X&=&\set{(\pi((1-\tau)u+\tau e))_t:u\in A_0,\tau\in[0,1], 0\leq t\leq R},\\
        A&=&\set{u_t:u\in A_0, 0\leq r\leq R}\cup\set{(\pi((1-\tau)u+\tau e))_R:u\in A_0, \tau\in[0,1]},\\
        B&=&\set{u_\rho:u\in B_0}.
    \end{eqnarray*}
    If $u\in A_0$, then \eqref{27-1} and \eqref{112} gives
    \[
    \Psi(u_t)\leq-\frac{t^{\sigma r-N}}{r}\om|u|^rdx,\quad\forall t>0,
    \]
    which implies that
   \begin{equation}\label{27-3}
       \sup_{u\in A_1}\Phi(u)\leq0,
   \end{equation}
    where $A_1=\set{u_t:u\in A_0, 0\leq r\leq R}$, for any $R>0$. Also, it is easy to see that since $A_0$ is compact, there exists $R>0$ large enough so that
    \begin{equation}\label{27-4}
         \sup_{u\in A_2}\Phi(u)\leq0,
    \end{equation}
    where $A_2=\set{(\pi((1-\tau)u+\tau e))_R:u\in A_0, \tau\in[0,1]}$.
    If $u\in B_0$, \eqref{27-2} implies
    \begin{equation*}
        \Phi(u_t)\geq t^s\seq{1-\frac{\lambda}{\lambda_{k+1}}+o(1)}\quad\text{as }t\to0,
    \end{equation*}
    uniformly on $B_0$. Therefore there exists $\rho>$ small enough so that
    \begin{equation}\label{27-5}
        \inf_{u\in B}\Phi(u)>0.
    \end{equation}
    Combining \eqref{27-3}-\eqref{27-5} we have the first inequality in \eqref{scalin1}. From the fact that $X$ is a compact set, the second inequality in \eqref{scalin1} also holds. Since $\Phi$ satisfies the $(PS)$ condition from Lemma \eqref{Lemma3}, Theorem \ref{scaled linking theorem} gives us a nontrivial critical point of $\Phi$.
\end{proof}

\begin{proof}[Proof of Theorem \ref{Theorem 8}]
    In this proof we will also employ Theorem \ref{scaled linking theorem}. Denoting by $\Phi$ the variational functional associated with our problem, note that it satisfies the $(PS)$ condition by Lemma \ref{Lemma3}, and we also have
    \[
\Phi(u)=I_s(u)-\lambda J_s(u)-\mu\om G(|x|,u)dx-\frac{1}{r}\om|u|^rdx.
\]
In particular,
\begin{eqnarray}\label{28-1}
    \Phi(u_t)=t^s\seq{1-\frac{\lambda}{\w{\Psi}(u)}}-\mu\om G(|x|,u_t)dx-\frac{t^{\sigma r-N}}{r}\om|u|^rdx\quad\forall u\in\M_s,
\end{eqnarray}

    As in the proof of Theorem \ref{Theorem 6}, it is easy to see that if $\lambda<\lambda_1$, then $\Phi$ has the mountain pass geometry. So we can assume that $\lambda_k<\lambda<\lambda_{k+1}$ for some $k\geq1$. Since $\M_s\backslash\w{\Psi}_\lambda$ is an open symmetric set of index $k$ by Theorem \ref{Theorem 301} \ref{301-3}, it has a compact symmetric subset $A_0$ of idex $k$ (see the proof of Proposition 3.1 in Degiovanni and Lancelotti \cite{MR2371112}). Let $B_0=\w{\Psi}_{\lambda_{k+1}}$, hence $i(\M_s\backslash B_0)=k$. 

   As in Theorem \ref{scaled linking theorem}, let $R>\rho>0$, let $e\in\M\backslash A_0$, and let
    \begin{eqnarray*}
        X&=&\set{(\pi((1-\tau)u+\tau e))_t:u\in A_0,\tau\in[0,1], 0\leq t\leq R},\\
        A&=&\set{u_t:u\in A_0, 0\leq r\leq R}\cup\set{(\pi((1-\tau)u+\tau e))_R:u\in A_0, \tau\in[0,1]},\\
        B&=&\set{u_\rho:u\in B_0}.
    \end{eqnarray*}
    Since $\set{\pi((1-\tau)u+\tau e):u\in A_0, \tau\in[0,1]}$ is compact, it is easy to see that there exists $R>0$ large enough such that
   \begin{equation}\label{28-2}
       \sup_{u\in A_1}\Phi(u)\leq0,
   \end{equation}
   where $A_1=\set{(\pi((1-\tau)u+\tau e))_R:u\in A_0, \tau\in[0,1]}.$
   
    Since $A_0$ is a compact subset of $\M_s$, and the latter is bounded away from zero, we have that $\om|u|^rdx$ is bounded away from zero on $A_0$.
    This together with \eqref{111} gives
    \[
    \Phi(u_t)\leq|\mu| c_1 t^{\sigma q_3-N}+|\mu| c_2t^{\sigma q_4-N}-c_3t^{\sigma r-N}
    \]
    for all $u\in A_0$ and $t\geq0$ for some constants $c_i>0$, $i=1,2,3$. The above inequality gives us
   \begin{eqnarray*}
       \Phi(u_t)&\leq&|\mu| c_
    1t^{\sigma q_3-N}-(c_3/2)t^{\sigma r-N}+|\mu| c_2t^{\sigma q_4-N}-(c_3/2) t^{\sigma r-N}\nonumber\\
    &\leq&\max_{t\geq0}\set{|\mu| c_1t^{\sigma q_3-N}-(c_3/2)t^{\sigma r-N}}+\max_{t\geq0}\set{|\mu| c_1t^{\sigma q_3-N}-(c_3/2) t^{\sigma r-N}}\nonumber\\
    &=&\w{c}_1|\mu|^{(\sigma r-N)/\sigma(r-q_3)}+\w{c}_2|\mu|^{(\sigma r-N)/\sigma(r-q_4)},
    \end{eqnarray*}
    for all $t\geq0$, $u\in  A_0$, and some constants $\w{c}_i>0$, $i=1,2$. Consequently
    \begin{equation}\label{28-3}
        \sup_{u\in A_2}\Phi(u)\leq\w{c}_1|\mu|^{(\sigma r-N)/\sigma(r-q_3)}+\w{c}_2|\mu|^{(\sigma r-N)/\sigma(r-q_4)},
    \end{equation}
    where $A_2=\set{u_t:u\in A_0, 0\leq r\leq R}$.

   Now, by \eqref{111} and \eqref{28-1}, we have
   \[
     \Phi(u_t)\geq t^s\seq{1-\frac{\lambda}{\lambda_{k+1}}+o(1)}\quad\text{as }t\to0
   \]
   uniformly on $B_0$, consequently there exists $\rho>0$ small enough so that
   \begin{equation}\label{28-4}
       \inf_{u\in B}\Phi(u)>0.
   \end{equation}

   Combining \eqref{28-2}-\eqref{28-4} we get that the first inequality in \eqref{scalin1} holds if $|\mu|$ is small enough. We also have the second inequality in \eqref{scalin1} since $X$ is compact. Theorem \ref{scaled linking theorem} now gives us a nontrivial critical point of $\Phi$.
   \end{proof}

   \subsection{Proof of Theorem \ref{Theorem 9} and Theorem \ref{Theorem 10}}

   Here $N\geq 3$. The energy functional associated with problem \eqref{114} is given by

   \begin{eqnarray*}
        \Phi(u)=\frac{1}{2}\dev{u}+\frac{1}{2p}\col{u}-\frac{\lambda}{q}\sca{u}\nonumber\\
        -\frac{\mu}{\beta}\om|u|^\beta dx-\frac{1}{2^*}\cri{u}.
   \end{eqnarray*}
As it is customary for many nonlinear elliptic PDEs, the following lemma provides a Pohozaev type identity for the critical points of $\Phi.$

\begin{lemma}\label{pohozaev identity}
Let $\lambda,\mu\in \R$ and $N, \alpha, p, \beta, q$ as in Theorem \ref{Theorem 9} and Theorem \ref{Theorem 10} and $u\in E^{\alpha,p}(\R^N)$ be a critical point of the functional $\Phi.$ Then, $u$ satisfies the identity
\begin{eqnarray*}
         \dfrac{N-2}{2}\dev{u}+\dfrac{N+\alpha}{2p}\col{u}-\dfrac{N\lambda}{q}\om|u|^qdx\nonumber\\
         -\frac{N\mu}{\beta}\om|u|^\beta dx-\dfrac{N}{2^*}\cri{u}=0.
     \end{eqnarray*}
\end{lemma}
\begin{proof}
Note that $u$ satisfies 
\begin{equation}\label{for reg}-\Delta u= V(x)u, \quad \quad \textrm{in} \,\, \mathcal D'(\R^N)
\end{equation}
where we have set $V=V^+-V^-,$ 
\begin{equation*}V^+=|u|^{2^*-2}+\mu |u|^{\beta-2}+\lambda |u|^{q-2}
\end{equation*}
and 
\begin{equation*}
V^-=\Big(I_\alpha \star |u|^p\Big)|u|^{p-2}.
\end{equation*}
Here we are assuming $\lambda,\mu>0,$ the other cases are similar, as what is to follow is not sensitive to their sign. 
By Kato's inequality we have

\begin{equation}\label{for reg2}
-\Delta |u|\leq -\Delta |u|+V^-|u|\leq V^+|u|,   \quad \quad \textrm{in} \,\, \mathcal D'(\R^N).
\end{equation}

Observe that by the range of the exponents $V^+\in L^{N/2}_{\textrm{loc}}(\R^N).$  Performing on \eqref{for reg2} classical bootstrap, see e.g. Struwe \cite[Lemma B.3]{MR1411681} (and also Trudinger \cite{MR240748}, and \cite[Remark 5.3]{MR3568051}) gives us that $u\in L^\sigma_{\textrm{loc}}(\R^N),$ for all $\sigma<\infty.$ By Riesz potential estimates, see e.g. \cite{MR1717817} (see also \cite{MR3436232}), we now infer that $I_\alpha \star |u|^p$, and therefore $Vu$ in \eqref{for reg}, are both in $L^\sigma_{\textrm{loc}}(\R^N)$ for all $\sigma <\infty.$ By the classical Calder\'on-Zygmund theorem, this implies that $u\in W^{2,\sigma}_{\textrm{loc}}(\R^N)$ for all $\sigma<\infty$ (and by Morrey's estimates that $u\in C^{1,\gamma}_{\textrm{loc}}(\R^N)).$ This regularity is enough to carry out the usual integration by parts argument after multiplying the equation by a suitable truncation of $x\cdot \nabla u(x),$ and this concludes the proof. 
\end{proof}

We now study the PS condition. We start with the following

  \begin{lemma}\label{lemma0}
    Assume $1/p>(1/2-1/\alpha)_+$. Let $(u_j)$ be a sequence in $E^{\alpha, p}(\R^N)$ such that $u_j\weak u$ for some $u\in E^{\alpha, p}(\R^N)$. Then
    \[
    \om(I_{\alpha/2}\star|u_j|^p)(I_{\alpha/2}\star|u_j|^{p-2}u_jv)dx\to\om(I_{\alpha/2}\star|u|^p)(I_{\alpha/2}\star|u|^{p-2}uv)dx
    \]
    for any $v\in E^{\alpha, p}(\R^N)$.
    \end{lemma}
\begin{proof}
     First note that, from the semigroup property of the Riez potential, it is equivalent to prove 
  \[
  \int_{\R^N}\int_{\R^N}\dfrac{|u_j(x)|^p|u_j(y)|^{p-2}u_j(y)v(y)}{|x-y|^{N-\alpha}}dxdy\to\int_{\R^N}\int_{\R^N}\dfrac{|u(x)|^p|u(y)|^{p-2}u(y)v(y)}{|x-y|^{N-\alpha}}dxdy.
  \]
  By Mercuri et al. \cite[Proposition 3.3 and Proposition 4.5]{MR3568051} for a renamed subsequence we have 
\begin{align*}
&u_j \rightarrow u, \qquad \textrm{in}\,\,L^p_{\textrm{loc}}(\R^N)\\
&u_j \rightarrow u, \qquad \textrm{a.e. on}\,\,\R^N\\
&I_{\alpha/2}\star|u_j|^p\weak I_{\alpha/2}\star|u|^p \,\,\textrm{in}\,\,L^2(\R^N).
\end{align*}
Set $$F_j=\Big| |u_j(x)|^p|u_j(y)|^{p-2}u_j(y)v(y)-|u(x)|^p|u(y)|^{p-2}u(y)v(y)\Big|, \qquad d\mu=I_\alpha(x-y)dxdy.$$

By Young's inequality we have, for every $\varepsilon>0$
\begin{align*}
    F_j&=\Big| |u_j(x)|^p|u_j(y)|^{p-2}u_j(y)v(y)-|u(x)|^p|u(y)|^{p-2}u(y)v(y)\Big|\\
    &\leq\varepsilon^{p/(p-1)}|u_j(x)|^p|u_j(y)|^p+\varepsilon^{-p}|u_j(x)|^p|v(y)|^p+\Big||u(x)|^p|u(y)|^{p-2}u(y)v(y)\Big|=:G_j.
\end{align*}
Applying Fatou's lemma to $G_j-F_j\geq0$ we obtain
\begin{align*}
    &\varepsilon^{p/(p-1)}\int_{\R^{2N}}|u(x)|^p|u(y)|^pd\mu+\varepsilon^{-p}\int_{\R^{2N}}|u(x)|^p|v(y)|^pd\mu\\
    &\leq\varepsilon^{p/(p-1)}\liminf_{j\to\infty}\int_{\R^{2N}}|u_j(x)|^p|u_j(y)|^pd\mu+\varepsilon^{-p}\lim_{j\to\infty}\om(I_{\alpha/2}\star|u_j|^p)(I_{\alpha/2}\star|v|^p)dx\\
    &-\limsup_{j\to\infty}\int_{\R^{2N}} F_jd\mu.
\end{align*}
Noting the cancellation due to the identity
\[
\int_{\R^{2N}}|u(x)|^p|v(y)|^pd\mu=\om(I_{\alpha/2}\star|u|^p)(I_{\alpha/2}\star|v|^p)dx=\lim_{j\to\infty}\om(I_{\alpha/2}\star|u_j|^p)(I_{\alpha/2}\star|v|^p)dx,
\]
and since $(u_j)_{j\in\N}$ is bounded in $\e$, we finally obtain, for some uniform constant $C>0$
\[
\limsup_{j\to+\infty}\int_{\R^{2N}}F_jd\mu\leq C\varepsilon^{p/(p-1)},
\]
and this concludes the proof.
\end{proof}

We now prove a local $(PS)$ condition for $\Phi$. Let
\begin{eqnarray}\label{best Sobolev constant}
    S=\inf_{u\in\D^{1,2}(\R^N)\backslash\set{0}}\dfrac{\dev{u}}{\seq{\cri{u}}^{1/2^*}}
\end{eqnarray}
be the best Sobolev constant.
\begin{lemma}\label{ps condition critical case}
    Assume $p<N/(N-2)$,  $\lambda>0$, and $\beta\in[q,2^*)$ hold, then $\Phi$ satisfies the $(PS)_c$ condition for all $0<c<\frac{1}{N}S^{N/2}$.
\end{lemma}

\begin{proof}
    Let $(u_n)\subset \e$ be a $(PS)_c$ sequence of $\Phi$. Then
    \begin{eqnarray}\label{ps1}
        \dfrac{1}{2}\dev{u_j}+\dfrac{1}{2p}\col{u_j}-\dfrac{\lambda}{q}\om|u_j|^qdx
        -\frac{\mu}{\beta}\om|u_j|^\beta dx\nonumber\\-\dfrac{1}{2^*}\om|u_j|^{2^*}dx=c+o(1)
    \end{eqnarray}
    and
    \begin{eqnarray}\label{ps2}
        \om\nabla u_j\cdot\nabla vdx+\om(I_{\alpha/2}\star|u_j|^p)(I_{\alpha/2}\star|u_j|^{p-2}u_jv)dx-\lambda\om|u_j|^{q-2}u_jv\nonumber\\
        -\mu\om|u_j|^{\beta-2}u_jv-\om|u_j|^{2^*-2}u_jv=o(\|v\|)
    \end{eqnarray}
    for all $v\in\e$. Taking $v=u_j$ in $\eqref{ps2}$ gives
    \begin{eqnarray}\label{k1}
        \dev{u_j}+\col{u_j}-\lambda\om|u_j|^qdx-\mu\om|u_j|^\beta dx\nonumber\\
        -\om|u_j|^{2^*}dx=o(\|u_j\|).
    \end{eqnarray}
    First we show that $(u_j)$ is bounded in $\e$. Suppose $\|u_j\|\to\infty$ for a renamed subsquence. Set
    \begin{eqnarray*}
        t_j=t_{u_j}=\dfrac{1}{I_s(u_j)^{1/s}},\quad\w{u}_j=(u_j)_{t_j},\quad\w{t}_j=\dfrac{1}{t_j}=I_s(u_j)^{1/s}.
    \end{eqnarray*}
    Then $\w{u}_j\in\M_s$ and 
    \begin{eqnarray*}
        u_j=(\w{u}_j)_{\w{t}_j}=\w{t}_j^\sigma\w{u}_j(\w{t}_j\cdot).
    \end{eqnarray*}
    Since $\M_s$ is a bounded manifold, $(\w{u}_j)$ is bounded. Since $\|u_j\|\to\infty$, $I_s(u_j)\to\infty$ and hence $\w{t}_j\to\infty$. Note that
    \begin{eqnarray*}
        \|u_j\|=\|(\w{u}_j)_{\w{t}j}\|=O(\w{t}_j^{s/2}),
    \end{eqnarray*}
     then $\eqref{ps1}$ and $\eqref{ps2}$ can be written as 
     \begin{eqnarray}\label{ps3}
         \w{t}_j^s\seq{\dfrac{1}{2}\dev{\w{u}_j}+\dfrac{1}{2p}\col{\w{u}_j}}=\lambda\dfrac{\w{t}_j^{s}}{q}\om|\w{u}_j|^qdx\nonumber\\
         +\frac{\mu\w{t}_j^{\sigma \beta-N}}{\beta}\om|\w{u}_j|^\beta dx+\dfrac{\w{t}_j^{s\frac{N}{N-2}}}{2^*}\cri{\w{u}_j}+c+o(1)
     \end{eqnarray}
     and
     \begin{eqnarray}\label{ps4}
          \w{t}_j^s\seq{\dev{\w{u}_j}+\col{\w{u}_j}}=\lambda\w{t}_j^{s}\om|\w{u}_j|^qdx \nonumber\\
         +\mu \w{t}_j^{\sigma\beta-N}\om|\w{u}_j|^\beta dx+\w{t}_j^{s\frac{N}{N-2}}\cri{\w{u}_j}+o(\w{t}_j^{s/2})
     \end{eqnarray}
     respectively. Since $\w{t}_j\to\infty$, dividing $\eqref{ps4}$ by $\w{t}_j^s$ gives
     \begin{eqnarray*}
         \dev{\w{u}_j}+\col{\w{u}_j}=\lambda\om|\w{u}_j|^qdx+\w{t}_j^{\frac{2}{N-2}s}\om|\w{u}_j|^{2^*}dx\\
         +\mu\w{t_j}^{\sigma (\beta-q)}\om|\w{u}_j|^\beta dx+o(1)
     \end{eqnarray*}
     Since the left-hand side is bounded and $\lambda>0$, this in turn gives 
     \begin{eqnarray*}
         \cri{\w{u}_j}=O(\w{t}_j^{-\frac{2}{N-2}s})
     \end{eqnarray*}
     Since $(\w{u}_j)$ is bounded in $L^l(\R^N)$ for all $l\in(\p,q)$, the interpolation inequality
     \begin{eqnarray*}
         \om|\w{u}_j|^\beta dx\leq|\w{u}_j|_l^{l(1-\theta)}|\w{u}_j|_{2^*}^{2^*\theta},
     \end{eqnarray*}
     where $\beta=(1-\theta)l+\theta2^*$, implies
     \begin{eqnarray}\label{ps5}
         \w{t}_j^{\sigma(\beta-q)}\om|\w{u}_j|^\beta dx=O(\w{t}_j^{-\frac{2}{N-2}s\frac{\beta-l}{2^*-l}+\sigma(\beta-q)}),
     \end{eqnarray}
     Note that the fact that
     \begin{eqnarray*}
         -\frac{2}{N-2}s\frac{\beta-l}{2^*-l}+\sigma(\beta-q)<0
     \end{eqnarray*}
     is equivalent to $\beta<2^*$. This together with $\eqref{ps5}$ implies 
     \begin{eqnarray*}
         \w{t}_j^{\sigma(\beta-q)}\om|\w{u}_j|^\beta dx\to0
     \end{eqnarray*}
     as $j\to\infty.$ By interpolation, one also sees that $\om|\w{u}_j|^qdx\to0$. Now multiplying $\eqref{ps3}$ by $2^*$, subtracting $\eqref{ps4}$, and dividing by $\w{t}_j^s$ gives
     \begin{eqnarray*}
         \seq{\frac{2^*}{2}-1}\dev{\w{u}_j}+\seq{\frac{2^*}{2p}-1}\col{\w{u}_j}=o(1).
     \end{eqnarray*}
     This implies $\w{u}_j\to0$ in $\e$, contradicting the fact that $\w{u}_j\in\M_s$.

     Since $(u_j)$ is bounded and $\e$ is a reflexive Banach space, a renamed subsequence of $(u_j)$ converges weakly to some $u\in\e$. Then $u_j$ also converges to $u$ strongly in $L^l(\R^N)$ for all $l\in(\p,2^*)$, weakly in $L^{2^*}(\R^N)$, and a.e. in $\R^N$ for a further subsequence. Now passing to the limit in $\Phi^\prime(u_j)u$ and using Lemma $\ref{lemma0}$, we have
     \begin{eqnarray}\label{ps6}
         \dev{u}+\col{u}-\lambda\om|u|^qdx\nonumber\\-\mu\om|u|^\beta dx-\cri{u}=0,
     \end{eqnarray}
     and $u$ also satisfies the Phozaev identity
     \begin{eqnarray}\label{ps7}
         \dfrac{N-2}{2}\dev{u}+\dfrac{N+\alpha}{2p}\col{u}-\dfrac{N\lambda}{q}\om|u|^qdx\nonumber\\
         -\frac{N\mu}{\beta}\om|u|^\beta dx-\dfrac{N}{2^*}\cri{u}=0.
     \end{eqnarray}
     by Lemma \ref{pohozaev identity}. Set $v_j=u_j-u$. We will show that $v_j\to0$ in $\e$ for a renamed subsequence. By the classical Brezis-Lieb lemma and Mercuri et al.\cite[Proposition 4.1]{MR3568051}, 
     \begin{eqnarray}\label{ps8}
         \dev{u_j}-\dev{u}=\dev{v_j}+o(1),
     \end{eqnarray}
     \begin{eqnarray}\label{ps9}
         \col{u_j}-\col{u}\geq\col{v_j}+o(1)
     \end{eqnarray}
     and
     \begin{eqnarray}\label{ps10}
         \cri{u_j}-\cri{u}=\cri{v_j}+o(1)
     \end{eqnarray}

     Subtracting $\eqref{ps6}$ from $\eqref{k1}$ and combining with $\eqref{ps8}-\eqref{ps10}$ gives
     
     \begin{eqnarray}\label{ps11}
         \dev{v_j}+\col{v_j}\leq\cri{v_j}+o(1)\nonumber\\
         \leq S^{-2^*/2}\seq{\dev{v_j}}^{2^*/2}+o(1).
     \end{eqnarray}
     So it suffices to show that $\dev{v_j}\to0$ for a renamed subsequence. Suppose this is not the case. Then $\eqref{ps11}$ gives
     \begin{eqnarray}\label{ps12}
         \dev{v_j}\geq S^{N/2}+o(1).
     \end{eqnarray}
     Dividing $\eqref{k1}$ by $2^*$ and subtracting from $\eqref{ps1}$ gives

     \begin{eqnarray}\label{ps13}
         c=\frac{1}{N}\dev{u_j}+\seq{\frac{1}{2p}-\frac{1}{2^*}}\col{u_j}\nonumber\\-\lambda\seq{\frac{1}{q}-\frac{1}{2^*}}\om|u|^qdx-\mu\seq{\frac{1}{\beta}-\frac{1}{2^*}}\om|u|^\beta dx+o(1).
     \end{eqnarray}
     Since
     \begin{eqnarray*}
         \dev{u_j}\geq S^{N/2}+\dev{u}+o(1)
     \end{eqnarray*}
     by $\eqref{ps8}$ and $\eqref{ps12}$, and
     \begin{eqnarray*}
         \col{u_j}\geq\col{u}+o(1)
     \end{eqnarray*}
     by $\eqref{ps9}$, $\eqref{ps13}$ gives
     \begin{eqnarray*}
         c\geq\frac{1}{N}S^{N/2}+\frac{1}{N}\dev{u}+\seq{\frac{1}{2p}-\frac{1}{2^*}}\col{u}\\-\lambda\seq{\frac{1}{q}-\frac{1}{2^*}}\om|u|^qdx-\mu\seq{\frac{1}{\beta}-\frac{1}{2^*}}\om|u|^\beta dx
         \end{eqnarray*}
     Multiplying $\eqref{ps6}$ by $\frac{1}{N}+\frac{N-2}{2s}$ and $\eqref{ps7}$ by $-\frac{1}{s}$ adding them together and then subtracting it from the above inequality gives
     \begin{eqnarray}\label{ps14}
         c\geq\dfrac{1}{N}S^{N/2}+\mu\seq{\frac{1}{N}+\frac{N-2}{2s}-\frac{N}{\beta s}-\frac{1}{\beta}+\frac{1}{2^*}}\om|u|^\beta dx
         +\frac{1}{N}\cri{u}.
     \end{eqnarray}
     Note that
     \begin{eqnarray*}
         \seq{\frac{1}{N}+\frac{N-2}{2s}-\frac{N}{\beta s}-\frac{1}{\beta}+\frac{1}{2^*}}\geq0
     \end{eqnarray*}
     if and only if $\beta\geq2\, (2p + \alpha)/(2 + \alpha)$.
     This together with $\eqref{ps14}$ and $\mu>0$ implies that $c\geq\frac{1}{N}S^{N/2}$, which is a contradiction.
     \end{proof}

     \begin{proof}[Proof of Theorem \ref{Theorem 9}]
        With Lemma \ref{ps condition critical case} at hand, we will employ Theorem \ref{m1} with $c^*=S^{N/2}/N$. Let $\varepsilon\in(0,\lambda_{k+m}-\lambda_{k+m-1})$. Then
        \[
        i(\M_s\backslash\w{\Psi}_{\lambda_{k+m-1+\varepsilon}})=k+m-1
        \]
        by Theorem \ref{Theorem 301} \ref{301-3}. Since $\M_s\backslash\w{\Psi}_{\lambda_{k+m-1}+\varepsilon}$ is an open symmetric set of index $k+m-1$, then it has a compact symmetric subset $C$ with $i(C)=k+m-1$ (see the proof of Proposition 3.1 in Degiovanni and Lancelotti \cite{MR2371112}). We will apply Theorem \ref{m1} with $A_0=C$ and $B_0=\w{\Psi}_{k}$. If $\lambda_1=\cdots=\lambda_k$ then $B_0=\w{\Psi}_{\lambda_k}=\M_s$ and consequently
        \[
        i(\M_s\backslash B_0)=0\leq k-1.
        \]
        by Proposition \ref{Proposition 300}. If $\lambda_{l-1}<\lambda_{l}=\lambda_k=\cdots=\lambda_{k+m-1}$ for some $2\leq l\leq k$, then
        \[
        i(\M_s\backslash B_0)=i(\M_s\backslash\w{\Psi}_{l})=l-1\leq k-1
        \]
        by Theorem \ref{Theorem 301} \ref{301-3}. 

         Let $R>\rho>0$ and let
    \begin{eqnarray*}
        X&=&\set{u_t:u\in A_0, 0\leq t\leq R},\\
        A&=&\set{u_R:u\in A_0},\\
        B&=&\set{u_{\rho}:u\in B_0}.
        \end{eqnarray*}
        For $u\in\M_s$ and $t\geq0$ we have
        \begin{equation}\label{210-1}
            \Phi(u_t)=t^s\seq{1-\frac{\lambda}{\w{\Psi}(u)}}-\frac{t^{\frac{N}{N-2}s}}{2^*}\om|u|^{2^*}dx.
        \end{equation}
        Since $\M_s$ is bounded, we have
        \[
        \Phi(u_t)\geq t^s\seq{1-\frac{\lambda}{\w{\Psi}(u)}+o(1)}\quad\text{as }t\to0 
        \]
        uniformly on $B_0$. This implies that there exists $\rho>0$ small enough so that
        \[
        \inf_{u\in B}\Psi(u)>0.
        \]
        Now note that, for any $u\in A_0\subset\M_s\backslash\w{\Psi}_{\lambda_{k+m-1}+\varepsilon}$,
        \[
        \frac{1}{q}\om|u|^qdx>\frac{1}{\lambda_{k+m-1}+\varepsilon}>\frac{1}{\lambda_{k+m}}
        \]
        since $\varepsilon<\lambda_{k+m}-\lambda_{k+m-1}$. Now using the fact that $\M_s$ is bounded in $L^l(\R^N)$ for all $u\in(\p,q)$ together with the interpolation inequality
        \[
        |u|_q^q\leq|u|_l^{l(1-\theta)}|u|_{2^*}^{2^*\theta}
        \]
        where $q=(1-\theta)l+\theta2^*$, we have that $|u|_{2^*}$ is bounded away from zero on $A_0$ for a constant that does not depend on $\varepsilon$. 

        \eqref{210-1} together with $\lambda_{k+m-1}=\lambda_k$ gives
        \[
        \Psi(u_t)=t^s\seq{1-\frac{\lambda}{\lambda_{k}+\varepsilon}}-c_1t^{\frac{N}{N-2}s}
        \]
        for some constant $c_1>0$, which yields to
        \[
        \sup_{u\in A}\Phi(u)\leq R^s-c_1R^{\frac{N}{N-2}s}\leq0
        \]
        if $R$ is sufficiently large. This implies the first inequality in \eqref{m1-2}. We also have
        \[
        \sup_{u\in X}\Phi(u)\leq\sup_{t\ge0}\set{t^s\seq{1-\frac{\lambda}{\lambda_{k}+\varepsilon}}-c_1t^{\frac{N}{N-2}s}}=\frac{2}{Nc_1^{(N-2)/2}}\seq{1-\frac{\lambda}{\lambda_{k}+\varepsilon}}^{N/2},
        \]
        which gives us the second inequality in \eqref{m1-2} if $\lambda$ is close enough of $\lambda_k$ and $\varepsilon$ is sufficiently small. The desired conclusion now follows from Theorem \ref{m1}.
     \end{proof}

\begin{proof}[Proof of Theorem \ref{Theorem 10}]
    Taking into account Lemma \ref{best Sobolev constant}, as in the preceding proof we will employ Theorem \ref{m1} with $c^*=S^{N/2}/N$. Let $\lambda\in\R$ and $m\geq1$. The variational functional associated with problem \eqref{114} is given by
    \[
    \Phi(u)=I_s(u)-\lambda J_s(u)-\frac{\mu}{\beta}\om|u|^\beta dx-\frac{1}{2^*}\cri{u}.
    \]
    In particular,
    \begin{equation}\label{212-1}
        \Phi(u_t)=t^s\seq{1-\frac{\lambda}{\w{\Psi}(u)}}-\frac{\mu t^{\sigma\beta-N}}{\beta}\om|u|^\beta dx-\frac{t^{\sigma2^* -N}}{2^*}\om|u|^{2^*}dx,
    \end{equation}
    for all $u\in\M_s$ and $t\geq0$. Now take $k\geq1$ large enough so that $\lambda<\lambda_{k+1}$ and $\lambda_{k+m-1}<\lambda_{k+m}$. Since $\M_s\backslash\w{\Psi}_{\lambda_{k+m}}$ is an open symmetric set with index $k+m-1$ by Theorem \ref{Theorem 301} \ref{301-3}, it has a compact symmetric subset $C$ of index $k+m-1$  (see the proof of Proposition 3.1 in Degiovanni and Lancelotti \cite{MR2371112}). We will employ Theorem \ref{m1} with $A_0=C$ and $B_0=\w{\Psi}_{\lambda_{k+1}}$. 

    Let $R>\rho>0$ and let
    \begin{eqnarray*}
        X&=&\set{u_t:u\in A_0, 0\leq t\leq R},\\
        A&=&\set{u_R:u\in A_0},\\
        B&=&\set{u_{\rho}:u\in B_0}.
        \end{eqnarray*}

    Note that if $u\in A_0$ then
    \[
    \frac{1}{q}\om|u|^qdx>\frac{1}{\lambda_{k+m}}
    \]
    Consequently, since $A_0$ is bounded on $L^l(\R^N)$ for all $l\in(\p,q)$, the interpolation inequalities
    \[
    |u|_q^q\leq|u|_l^{\theta_1 l}|u|_\beta^{(1-\theta_1)\beta},\quad |u|_q^q\leq|u|_l^{\theta_2l}|u|_{2^*}^{(1-\theta_2)2^*},
    \]
    where $q=\theta_1l+(1-\theta_1)\beta$ and $q=\theta_2 l+(1-\theta_2)2^*$ give that $\om|u|^\beta dx$ and $\cri{u}$ are bounded away from zero on $A_0$ by constants that depend only on $m$. It is also easy to see that $|\lambda|/\w{\Psi}(u)$ is also bounded from above on $A_0$. Equation \eqref{212-1} now gives us that
   \begin{equation}\label{212-2}
        \Phi(u_t)\leq c_1t^{s}-\mu c_2t^{\sigma\beta-N}-c_3t^{\sigma 2^*-N}
   \end{equation}
    for all $u\in A_0$, $t\geq0$ and some constants $c_i>0$, $i=1,2,3.$ The last inequality implies that 
    \[
    \sup_{u\in A}\Phi(u)\leq c_1 R^s-c_3R^{\sigma 2^*-N}\leq0
    \]
    for large $R>0$.

    Now, for all $u\in B_0$
    \[
    \Phi(u_t)\geq t^s\seq{1-\frac{\lambda}{\lambda_{k+1}}+o(1)}\quad\text{as }t\to0
    \]
    uniformly on $B_0$ since $B_0$ is a bounded set. This gives
    \[
    \inf_{u\in B}\Phi(u)>0
    \]
    for a sufficiently small $\rho>0$. This implies the first inequality in \eqref{m1-2}, and from \eqref{212-2} we also have
    \[
    \sup_{u\in X}\Phi(u_t)\leq\max_{t \geq0}\set{c_1t^s-\mu t^{\sigma\beta-N}}=\frac{\w{C}}{\mu^{s/\sigma(\beta-q)}},
    \]
    for some constant $\w{C}>0$, which implies the second inequality in \eqref{m1-2} if $\mu$ is large enough. The desired conclusion now follows from Theorem \ref{m1-2}.
\end{proof}

\section*{Acknowledgements}
Carlo Mercuri is a member of the group GNAMPA of Istituto Nazionale di Alta Matematica (INdAM).

\end{document}